\documentclass[12pt,a4paper]{amsart}

\usepackage[english]{babel}
\usepackage[utf8x]{inputenc}
\usepackage{amsfonts}
\usepackage{amssymb}
\usepackage{amsmath}
\usepackage{amsrefs}
\usepackage{mathrsfs}
\usepackage{graphicx}
\usepackage{framed}
\usepackage{mdframed}
\usepackage[colorinlistoftodos]{todonotes}
\usepackage{color}

\newtheorem{theorem}{Theorem}[section]
\newtheorem{lemma}[theorem]{Lemma}
\newtheorem{proposition}[theorem]{Proposition}

\theoremstyle{definition}
\newtheorem{definition}[theorem]{Definition}

\usepackage{dsfont}
\usepackage{stmaryrd}

\usepackage[scr=boondox]{mathalfa}

\newcommand{\Ad}{\mathrm{Ad}}
\DeclareMathOperator{\Hol}{\mathrm{Hol}}
\DeclareMathOperator{\Aut}{\mathrm{Aut}}
\DeclareMathOperator{\Iso}{\mathrm{I}}
\DeclareMathOperator{\Orth}{\mathrm{O}}

\DeclareMathOperator{\GLin}{\mathrm{GL}}
\DeclareMathOperator{\SLin}{\mathrm{SL}}
\DeclareMathOperator{\PGL}{\mathrm{PGL}}

\DeclareMathOperator{\PO}{\mathrm{PO}}

\newcommand{\MC}[1]{\omega_{{}_{#1}}}
\newcommand{\Lt}[1]{\operatorname{L}_{#1}}
\newcommand{\Rt}[1]{\operatorname{R}_{#1}}

\usepackage[all]{xy}

\usepackage[normalem]{ulem}

\usepackage[T2A,OT1]{fontenc}

\DeclareFontEncoding{T3}{}{}
\DeclareFontSubstitution{T3}{cmr}{m}{n}
\DeclareSymbolFont{tipa}{T3}{cmr}{m}{sl}
\SetSymbolFont{tipa}{bold}{T3}{cmr}{bx}{sl}

\DeclareMathSymbol{\kgf}{\mathord}{tipa}{'255}

\title[Sprawl prelude]{A method for determining \\ Cartan geometries from the local behavior of automorphisms}
\author{Jacob W. Erickson}
\thanks{Partially supported by the Brin Graduate Fellowship at the University of Maryland}
\date{\today}

\begin{document}
\begin{abstract}We introduce a construction for a Cartan geometry that captures the local behavior of a given geometric automorphism near a distinguished element. The result of this construction, which\linebreak we call the sprawl generated by the automorphism, is uniquely characterized by a kind of universal property that allows us to compare different Cartan geometries that admit automorphisms with equivalent local behavior near a distinguished element. As example applications, we describe how to construct non-flat real projective structures admitting nontrivial automorphisms with higher-order fixed points and extend some known local automorphisms with higher-order fixed points on non-flat parabolic geometries to global automorphisms.\end{abstract}

\maketitle


\section{Introduction}
The behavior of symmetries, when they exist, can often tell us a great deal about a particular geometric structure. For example, isometries of a Riemannian manifold necessarily act properly, so if the automorphism group of a conformal Riemannian structure were to act \textit{non}-properly, then it could not possibly preserve an underlying Riemannian metric in the conformal class. Moreover, a celebrated theorem of Ferrand and Obata tells us that, in each dimension greater than two, there are only two conformal structures of definite signature for which the automorphism group acts non-properly: the conformal sphere and the conformal structure overlying Euclidean space.

A natural setting for investigating the general behavior of symmetries\linebreak for various types of geometric structures is the unifying framework of Cartan geometries. Working with Cartan geometries, which extend the spirit of Klein's Erlangen program to a plethora of different types of geometric structures that are \textit{modelled} on particular homogeneous geometries, frequently leads to comprehensive, overarching results that apply in far greater generality than to just a single type of geometric structure. Indeed, in \cite{FrancesRankOneFO}, Frances recognized that the Ferrand-Obata theorem was a particular instance of a more general result applying to all parabolic Cartan geometries of real rank 1, under certain mild curvature restrictions. Even when we just restrict to specific types of geometric structures, though, the overarching framework given by the Cartan machinery is still quite well-suited to exploring the behavior of automorphisms. The recent work done in \cite{MelnickPecastaing2022} and \cite{FrancesMelnick2023} toward resolving\linebreak the Lorentzian Lichnerowicz conjecture, for example, was the result of careful consideration of the Cartan geometries canonically associated to conformal Lorentzian structures.

Given the utility and wide applicability of Cartan geometries and their automorphisms for studying symmetries of geometric structures, it would be useful to know the extent of what a given automorphism can tell us about a Cartan geometry. Toward this goal, this paper presents a method for determining precisely what we can learn about Cartan geometries from the \textit{local} geometric behavior of automorphisms.

Specifically, we construct a kind of ``universal example''---which we call a \emph{sprawl}---for a Cartan geometry admitting an automorphism with given local geometric behavior. Our main result, Theorem \ref{universalproperty}, then provides a kind of universal property for these objects, showing that there is a natural geometric map into each Cartan geometry admitting an automorphism with the generating local behavior coming from the sprawl. In other words, the sprawl construction completely encodes the information given by the local behavior of an automorphism on a particular open set of the geometry, putting it into a useful form.

\section*{Acknowledgements}
We would specifically like to thank Karin Melnick for several helpful suggestions and conversations, and Charles Frances, both for inspiring this paper and for tolerating the author's abundant enthusiasm for the topic. Additionally, we would like to thank the anonymous referees who reviewed earlier versions of this paper and managed to provide genuinely useful improvements.

\section{Preliminaries}\label{preliminaries}
The standard references for the fundamentals of Cartan geometries are \cite{Sharpe1997} and \cite{CapSlovakPG1}. In this section, we are merely specifying notation and terminology; we do not intend this as a first introduction to the topic.

\begin{definition}For a Lie group $G$ and closed subgroup $H$ such that $G/H$ is connected, we call the pair $(G,H)$ a \emph{model} or \emph{model geometry}. The Lie group $G$ is called the \emph{model group} and $H$ is called the \emph{isotropy} or \emph{stabilizer subgroup}.\end{definition}

For example, writing $\mathrm{Aff}(m):=\mathbb{R}^m\rtimes\GLin_m\mathbb{R}$ for the Lie group of affine transformations on $\mathbb{R}^m$ and thinking of $\GLin_m\mathbb{R}$ as the closed subgroup of $\mathrm{Aff}(m)$ fixing the origin $0\in\mathbb{R}^m$, the pair $(\mathrm{Aff}(m),\GLin_m\mathbb{R})$ is a model corresponding to affine geometry on $\mathrm{Aff}(m)/\GLin_m\mathbb{R}\cong\mathbb{R}^m$.

As the name should hopefully suggest, these \emph{model geometries} act as models for the various types of \emph{Cartan geometries}.

\begin{definition}For a model geometry $(G,H)$, a \emph{Cartan geometry of type $(G,H)$} over a smooth manifold $M$ is a pair $(\mathscr{G},\omega)$, where $\mathscr{G}$ is a principal $H$-bundle over $M$ and $\omega$ is a $\mathfrak{g}$-valued 1-form on $\mathscr{G}$ satisfying the following three criteria:
\begin{itemize}
\item For each $\mathscr{g}\in\mathscr{G}$, $\omega_\mathscr{g}:T_\mathscr{g}\mathscr{G}\to\mathfrak{g}$ is a linear isomorphism.
\item For each $h\in H$, $\Rt{h}^*\omega=\Ad_{h^{-1}}\omega$, where $\Rt{h}:\mathscr{g}\mapsto\mathscr{g}h$ denotes right-translation by $h$.
\item For each $Y\in\mathfrak{h}$, the flow of the vector field $\omega^{-1}(Y)$ is given by $\exp(t\omega^{-1}(Y))=\Rt{\exp(tY)}$ for all $t\in\mathbb{R}$.
\end{itemize}\end{definition}

To make the notation cleaner, we will always denote the quotient map of a principal $H$-bundle by $q_{{}_H}$, even when there are multiple principal $H$-bundles involved; the meaning should always be clear from context.

The geometric structure of a model geometry $(G,H)$, when encoded as a Cartan geometry, is called the \emph{Klein geometry of type $(G,H)$}.

\begin{definition}For a model $(G,H)$, the \emph{Klein geometry of type $(G,H)$} is the Cartan geometry of type $(G,H)$ over $G/H$ given by the pair $(G,\MC{G})$, where $G$ is the model group and $\MC{G}$ is the Maurer--Cartan form on $G$ given by $\MC{G}(X_g):=\Lt{g^{-1}*}X_g\in T_eG=\mathfrak{g}$ for $X_g\in T_gG$.\end{definition}

To compare different Cartan geometries of the same type, we will use \emph{geometric maps}.

\begin{definition}For two Cartan geometries $(\mathscr{G}_1,\omega_1)$ and $(\mathscr{G}_2,\omega_2)$ of type $(G,H)$, a \emph{geometric map} $\varphi:(\mathscr{G}_1,\omega_1)\to(\mathscr{G}_2,\omega_2)$ is an $H$-equivariant smooth map $\varphi:\mathscr{G}_1\to\mathscr{G}_2$ such that $\varphi^*\omega_2=\omega_1$.\end{definition}

For notational convenience, whenever a particular map or relation on the overlying principal $H$-bundles canonically induces a corresponding map or relation on the underlying base manifolds, we will always denote this induced map or relation by the same symbol as the bundle map or relation. A particularly common use of this will be for geometric maps, since they induce corresponding local diffeomorphisms between the base manifolds by $H$-equivariance; in short, for each geometric map $\varphi:(\mathscr{G}_1,\omega_1)\to(\mathscr{G}_2,\omega_2)$, we will write $\varphi(q_{{}_H}(\mathscr{g})):=q_{{}_H}(\varphi(\mathscr{g}))$ for the induced map on the base manifolds.

\textit{Injective} geometric maps will be particularly important to us. We will call a geometric map $\varphi:(\mathscr{G}_1,\omega_1)\to(\mathscr{G}_2,\omega_2)$ a \emph{geometric embedding} when $\varphi$ is injective and a \emph{(geometric) isomorphism} when $\varphi$ is bijective. Moreover, a geometric isomorphism from a Cartan geometry to itself is called a \emph{(geometric) automorphism}.

Automorphisms of Cartan geometries tend to be fairly rigid. Given an automorphism $\alpha$ of $(\mathscr{G},\omega)$ and an element $\mathscr{e}\in\mathscr{G}$, the image $\alpha(\mathscr{e})$ uniquely determines $\alpha$ when the base manifold is connected. The group $\Aut(\mathscr{G},\omega)$ of all automorphisms of $(\mathscr{G},\omega)$ therefore acts freely on $\mathscr{G}$, and we can induce a Lie group structure on it by looking at the smooth structure inherited from orbits of $\Aut(\mathscr{G},\omega)$ in $\mathscr{G}$.

Another tool for comparing different Cartan geometries is \emph{curvature}, which helps to locally distinguish Klein geometries from other Cartan geometries of the same type.

\begin{definition}The \emph{curvature} of a Cartan geometry $(\mathscr{G},\omega)$ is the $\mathfrak{g}$-valued 2-form $\Omega:=\mathrm{d}\omega+\tfrac{1}{2}[\omega,\omega]$.\end{definition}

The curvature of a Cartan geometry vanishes in a neighborhood of a point if and only if it is geometrically equivalent to the Klein geometry in a neighborhood of that point. In other words, $\Omega$ vanishes on some neighborhood of $\mathscr{e}\in\mathscr{G}$ if and only if there exists a geometric embedding \[\varphi:(q_{{}_H}^{-1}(U),\MC{G})\hookrightarrow(\mathscr{G},\omega)\] from an $H$-invariant neighborhood $q_{{}_H}^{-1}(U)$ of $e\in G$ to $\mathscr{G}$ such that $\varphi(e)=\mathscr{e}$.

Finally, our primary tool for working with Cartan geometries in this paper is that of \emph{development}.

\begin{definition}Given a (piecewise smooth)\footnote{Throughout this paper, a ``path'' will always mean a piecewise smooth path.} path $\gamma:[0,1]\to\mathscr{G}$ in a Cartan geometry $(\mathscr{G},\omega)$ of type $(G,H)$, its \emph{development} is the unique path $\gamma_G:[0,1]\to G$ in $G$ such that $\gamma_G(0)=e$ and $\omega(\dot{\gamma})=\omega_G(\dot{\gamma}_G)$, where $\omega_G$ is the Maurer--Cartan form on $G$.\end{definition}

The idea here is that the tangent vectors $\dot{\gamma}$ tell us how to move along $\gamma$ at each point in time, and $\gamma_G$ is the path we get by trying to follow these same instructions in the model group $G$, starting at the identity. Crucially, it follows that if we have two paths with the same development and starting point in a Cartan geometry, then they must be the same path.

\section{Sprawls}\label{sprawlsection}
We would like to construct Cartan geometries that are generated ``as freely as possible'' by the local behavior of an automorphism. We call such geometries \textit{sprawls}, a term chosen both to evoke the idea of something extending as lazily as possible, and to sound like the word \textit{span}, which plays a vaguely similar role for vector spaces.


To explain the ideas involved effectively, we start by giving the set-up of the construction and describing a na\"{i}ve approach to achieving what we want. While this na\"{i}ve approach ultimately does not work, it serves to motivate the considerably more complicated definition of the sprawl, which does exactly what we want it to do. After giving the appropriate definitions and verifying that they make sense, we will finally state and prove the key result of the paper, Theorem \ref{universalproperty}, which gives a kind of universal property for sprawls that will allow us to compare Cartan geometries admitting automorphisms with similar local behavior.

\subsection{The set-up and a na\"{i}ve approach}\label{sprawl1}
Throughout this section, let $(\mathscr{G},\omega)$ be a Cartan geometry of type $(G,H)$ over a connected smooth manifold $M$ with a distinguished element $\mathscr{e}\in\mathscr{G}$ and an automorphism $\alpha\in\Aut(\mathscr{G},\omega)$. Furthermore, we fix a connected open subset $U$ of $M$ containing both $q_{{}_H}(\mathscr{e})$ and $q_{{}_H}(\alpha(\mathscr{e}))$; this allows $U$ to capture the local behavior of $\alpha$ near $\mathscr{e}$, in the sense that sufficiently small open neighborhoods of $q_{{}_H}(\mathscr{e})$ will be mapped back into $U$ by $\alpha$.

Because $\alpha$ is an automorphism, all of the iterates of $q_{{}_H}^{-1}(U)$ under $\alpha$ are geometrically equivalent, but inside $\mathscr{G}$, they might glue together in ways that are unnecessary to still admit an automorphism that behaves like $\alpha$ near $\mathscr{e}$. As a simple example, consider the case where $(\mathscr{G},\omega)$ is the Riemannian geometry over a Euclidean torus, $\alpha$ is a translation, and $U$ is a small neighborhood of some point $q_{{}_H}(\mathscr{e})$: while successive iterates of $\alpha$ will push $U$ back around to itself, as in Figure \ref{torusexample}, lifting to the Euclidean plane demonstrates a situation with an automorphism exhibiting the same local behavior as $\alpha$, but which does not push (the geometrically identical copy of) $U$ back onto itself.

\begin{figure}
\centering\includegraphics[width=0.6\textwidth]{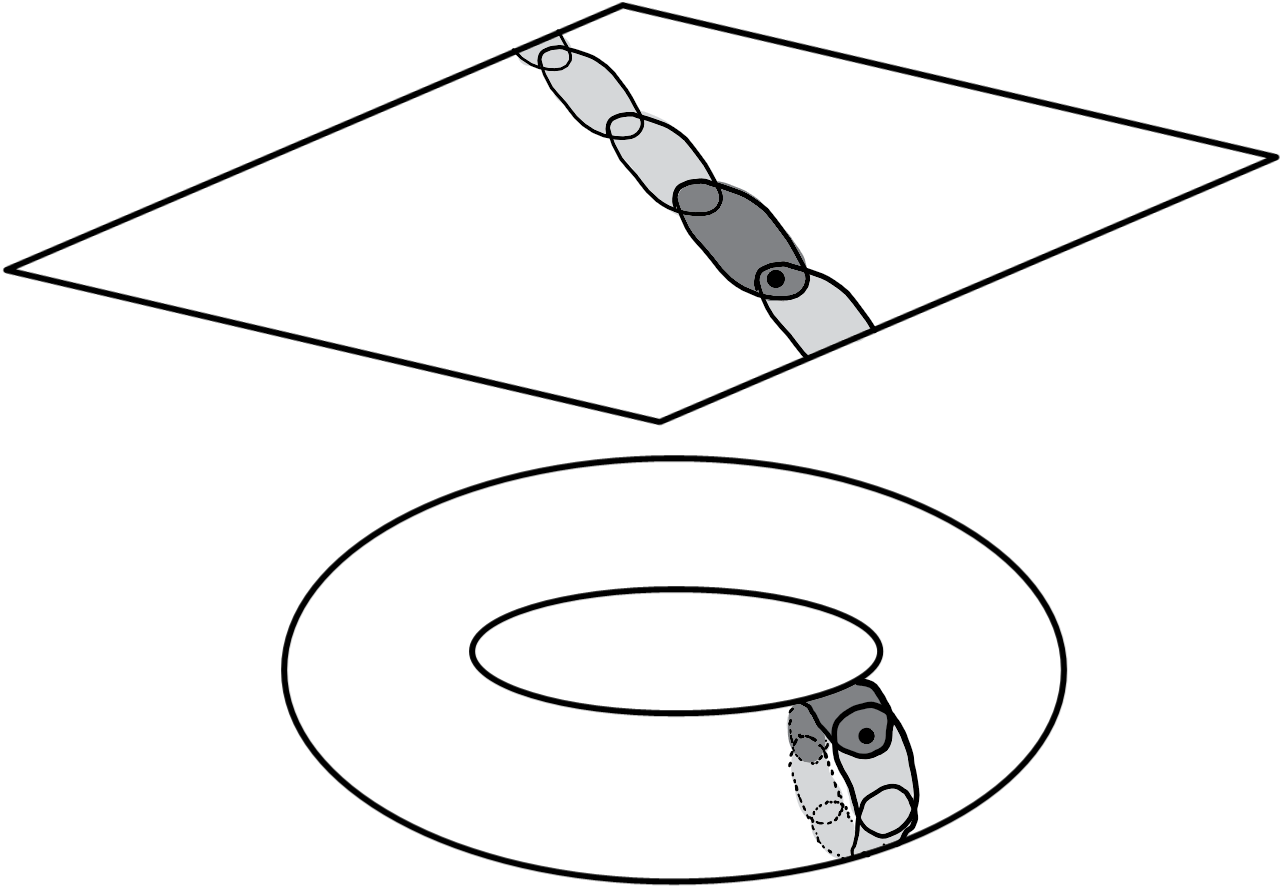}
\caption{The region $U$ (highlighted in darker gray) is pushed back to itself in the torus by iterates of the translation $\alpha$, but lifting the situation to the plane gives a situation with identical local behavior such that $U$ never returns to itself after leaving}
\label{torusexample}
\end{figure}

Our goal is, in essence, to construct a geometry that is generated ``as freely as possible'' by the local behavior of $\alpha$. In other words, we would like to construct a geometry by taking iterates of $U$ under $\alpha$ and gluing them together as little as possible to still retain an automorphism with the same local behavior as $\alpha$ near the distinguished point $\mathscr{e}$.

To specify these iterates in a way that avoids implicitly gluing them inside $\mathscr{G}$, we define, for each $i\in\mathbb{Z}$, a relabeling map \[\tilde{\alpha}^i:q_{{}_H}^{-1}(U)\to\tilde{\alpha}^i(q_{{}_H}^{-1}(U)),\] where $\tilde{\alpha}^i(q_{{}_H}^{-1}(U))$ is a diffeomorphic copy of $q_{{}_H}^{-1}(U)$ with all of its points $\mathscr{g}$ rewritten as $\tilde{\alpha}^i(\mathscr{g})$. There is a natural right $H$-action on $\tilde{\alpha}^i(q_{{}_H}^{-1}(U))$ given by, for each $h\in H$, $\tilde{\alpha}^i(\mathscr{g})h:=\tilde{\alpha}^i(\mathscr{g}h)$, which makes $\tilde{\alpha}^i$ an $H$-equivariant map and, therefore, an isomorphism of principal $H$-bundles.

With this notation, we can specify what we are doing a bit more concretely. We will take the disjoint union $\bigsqcup_{i\in\mathbb{Z}}\tilde{\alpha}^i(q_{{}_H}^{-1}(U))$ and apply some minimal gluing (via an equivalence relation $\sim$) to obtain a new Cartan geometry for which $\tilde{\alpha}:\tilde{\alpha}^i(\mathscr{g})\mapsto\tilde{\alpha}^{i+1}(\mathscr{g})$ is an automorphism with the same local behavior as $\alpha$ near $\mathscr{e}$. Identifying $\tilde{\alpha}^0(q_{{}_H}^{-1}(U))$ with $q_{{}_H}^{-1}(U)$, so that we may think of $\tilde{\alpha}^0(\mathscr{e})\in\tilde{\alpha}^0(q_{{}_H}^{-1}(U))$ as $\mathscr{e}\in q_{{}_H}^{-1}(U)$, this amounts to requiring $\tilde{\alpha}(\mathscr{e})=\alpha(\mathscr{e})$, since automorphisms of Cartan geometries over a connected base manifold are uniquely determined by their image on a single element.

If $\tilde{\alpha}^{i+1}(\mathscr{e})=\tilde{\alpha}^i(\tilde{\alpha}(\mathscr{e}))=\tilde{\alpha}^i(\alpha(\mathscr{e}))$, then for every (piecewise smooth) path $\gamma:[0,1]\to q_{{}_H}^{-1}(U\cap\alpha(U))$ starting with $\alpha(\mathscr{e})$, we must also have $\tilde{\alpha}^{i+1}(\alpha^{-1}(\gamma(t)))=\tilde{\alpha}^i(\gamma(t))$ for all $t\in[0,1]$, since $\tilde{\alpha}^{i+1}(\alpha^{-1}(\gamma))$ and $\tilde{\alpha}^i(\gamma)$ are paths with the same development and starting point. In other words, whatever this new Cartan geometry ends up being, adjacent iterates $\tilde{\alpha}^i(q_{{}_H}^{-1}(U))$ and $\tilde{\alpha}^{i+1}(q_{{}_H}^{-1}(U))$ must be glued together by identifying $\tilde{\alpha}^i(\mathscr{g})$ with $\tilde{\alpha}^{i+1}(\alpha^{-1}(\mathscr{g}))$ whenever $q_{{}_H}(\mathscr{g})$ lies within the same connected component of $U\cap\alpha(U)$ as $\alpha(q_{{}_H}(\mathscr{e}))$. With this in mind, it is tempting to imagine that the minimal equivalence relation on $\bigsqcup_{i\in\mathbb{Z}}\tilde{\alpha}^i(q_{{}_H}^{-1}(U))$ that accomplishes these necessary identifications between adjacent iterates is sufficient as well. Indeed, we can see that this gluing gives precisely the right answer in the torus example above. We will call this equivalence relation the \emph{na\"{i}ve gluing}.

\begin{definition}The \emph{na\"{i}ve gluing} $\sim_\text{naive}$ is the minimal equivalence relation on $\bigsqcup_{i\in\mathbb{Z}}\tilde{\alpha}^i(q_{{}_H}^{-1}(U))$ such that, for each $i\in\mathbb{Z}$, $\tilde{\alpha}^i(\mathscr{g})\sim_\text{naive}\tilde{\alpha}^{i+1}(\alpha^{-1}(\mathscr{g}))$ whenever $q_{{}_H}(\mathscr{g})\in U\cap\alpha(U)$ is contained in the same connected component of $U\cap\alpha(U)$ as $\alpha(q_{{}_H}(\mathscr{e}))$.

Consistent with our notational convention for canonically induced maps and relations on base manifolds from Section \ref{preliminaries}, we will use the same symbol $\sim_\text{naive}$ to denote the induced equivalence relation on $\bigsqcup_{i\in\mathbb{Z}}\tilde{\alpha}^i(U)$, given by $\tilde{\alpha}^{i_1}(q_{{}_H}(\mathscr{g}_1))\sim_\text{naive}\tilde{\alpha}^{i_2}(q_{{}_H}(\mathscr{g}_2))$ if and only if $\tilde{\alpha}^{i_1}(\mathscr{g}_1)\sim_\text{naive}\tilde{\alpha}^{i_2}(\mathscr{g}_2h)$ for some $h\in H$. We will also refer to this as the na\"ive gluing.\end{definition}

Unfortunately, this na\"{i}ve gluing will not work in general. To see this, consider the Klein geometry $(\Iso(2),\MC{\Iso(2)})$ of type $(\Iso(2),\Orth(2))$ over $\mathbb{R}^2$, corresponding to the Euclidean plane. Within this geometry, we choose a rotation $\alpha$ with infinite order that fixes $0$ and an open set $U$ given by the union of a small open ball centered on $0$ and an open sector of the plane that is disjoint from its image under $\alpha$, as depicted in Figure \ref{rotexample}. The identity element $(0,\mathds{1})$, which we take to be our distinguished element, is contained in $q_{{}_{\Orth(2)}}^{-1}(U)$, as is $\alpha(0,\mathds{1})$, since $\alpha$ fixes $0$. Under the na\"{i}ve gluing above, the iterates $\tilde{\alpha}^i(q_{{}_{\Orth(2)}}^{-1}(U))$ all coincide over the small open ball around $0$, but nowhere else. This becomes a problem whenever $U\cap\alpha^i(U)$ has points that lie outside of that small open ball: if $x\in U\cap\alpha^i(U)$ lies on the boundary of the open ball, then every neighborhood of $\tilde{\alpha}^i(\alpha^{-i}(x))$ must intersect every neighborhood of $\tilde{\alpha}^0(x)\cong x$ inside the open ball, so since $\tilde{\alpha}^i(\alpha^{-i}(x))$ is not identified with $x$ under the na\"{i}ve gluing, the resulting space is not even Hausdorff.

\begin{figure}
\centering\includegraphics[width=0.6\textwidth]{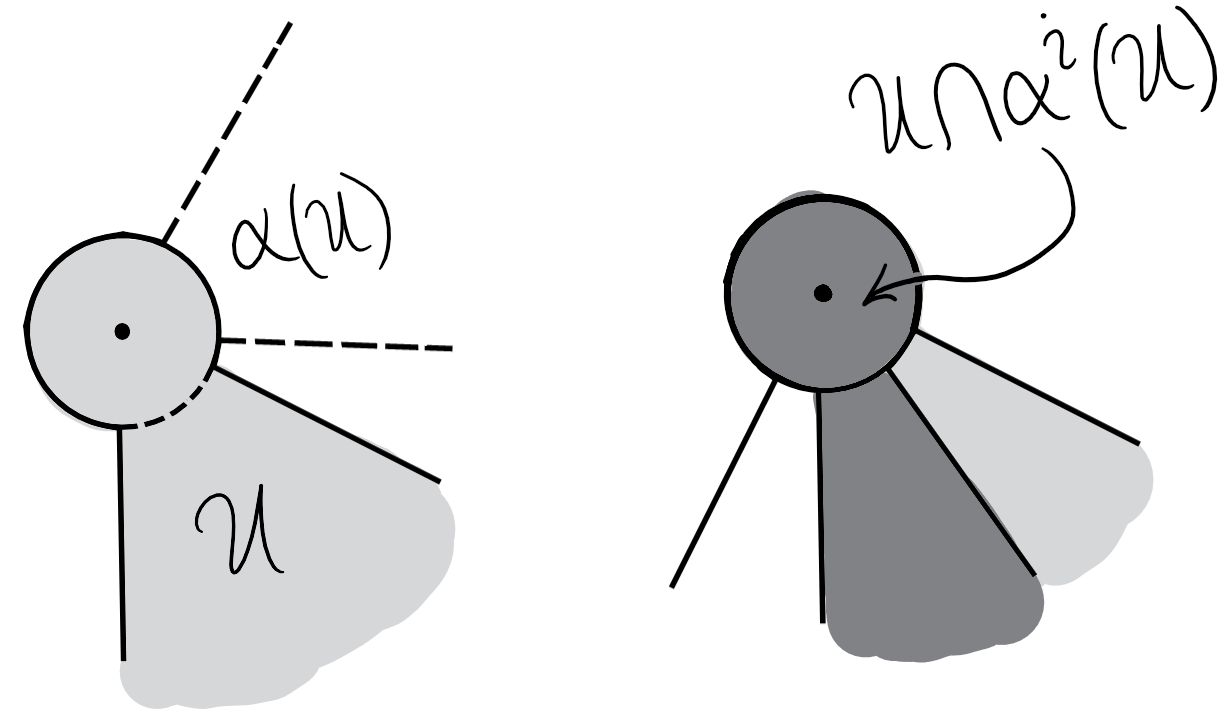}
\caption{The region $U$ (highlighted in lighter gray) given by the union of an open ball and an open sector that is disjoint from its image under the rotation $\alpha$, as well as a depiction of its intersection (highlighted in darker gray) with an iterate under $\alpha$ where the overlap escapes the open ball}
\label{rotexample}
\end{figure}

We can, fortunately, salvage this idea with some slightly intricate modifications. Consider a path $\gamma:[0,1]\to U\cap\alpha^i(U)$ that starts outside of the open ball and ends inside of it. Then, we get corresponding paths $\tilde{\alpha}^0(\gamma)\cong\gamma$ and $\tilde{\alpha}^i(\alpha^{-i}(\gamma))$ in $\tilde{\alpha}^0(U)$ and $\tilde{\alpha}^i(U)$, respectively, and we can lift these to paths $\hat{\gamma}_0$ in $\tilde{\alpha}^0(q_{{}_{\Orth(2)}}^{-1}(U))$ and $\hat{\gamma}_1$ in $\tilde{\alpha}^i(q_{{}_{\Orth(2)}}^{-1}(U))$ with the same development and endpoint. In particular, $\hat{\gamma}_0$ and $\hat{\gamma}_1$ must coincide inside the new Cartan geometry, if it exists, so that the concatenation $\hat{\gamma}_0\star\overline{\hat{\gamma}_1}$ of $\hat{\gamma}_0$ with the reverse of $\hat{\gamma}_1$ is a loop that ``backtracks'' over itself.

The new strategy, therefore, is to identify elements $\tilde{\alpha}^{i_1}(\mathscr{g}_1)$ and $\tilde{\alpha}^{i_2}(\mathscr{g}_2)$ whenever we can find a path starting at $\alpha^{i_1}(\mathscr{g}_1)$ that only crosses between iterates at points identified under the na\"{i}ve gluing and which ``backtracks'' over itself to end up at $\alpha^{i_2}(\mathscr{g}_2)$. In the next subsection, we will formalize this correction to the na\"{i}ve gluing, which we will use to define the sprawl.

\subsection{The definition of the sprawl}\label{sprawl2}
To start, we provide a way of describing paths that only cross between iterates at the points identified under the na\"{i}ve gluing.

\begin{definition}A \emph{$(U,\alpha,\mathscr{e})$-incrementation}\footnote{We will consistently drop reference to $U$, $\alpha$, and $\mathscr{e}$ when they are to be understood from context. For example, we will typically just refer to an \emph{incrementation}, rather than a $(U,\alpha,\mathscr{e})$-incrementation.} for $\gamma:[0,1]\to M$ is a finite partition $0=t_0<\cdots<t_\ell=1$ of $[0,1]$ together with a finite sequence of integers $k_0,\dots,k_{\ell-1}\in\mathbb{Z}$ such that, for each $0\leq j<\ell$, $|k_j-k_{j+1}|=1$ and $\gamma([t_j,t_{j+1}])\subseteq\alpha^{k_j}(U)$, and for each $0\leq j<\ell-1$, $\gamma(t_{j+1})$ is in the connected component of $\alpha^{k_j}(U)\cap\alpha^{k_{j+1}}(U)$ containing $q_{{}_H}(\alpha^{\max(k_j,k_{j+1})}(\mathscr{e}))$. The integers $k_0$ and $k_{\ell-1}$ are called the \emph{initial label} and \emph{terminal label}, respectively.\end{definition}

\begin{definition}We say a path $\gamma:[0,1]\to M$ is \emph{$(U,\alpha,\mathscr{e})$-incremented from $i_1$ to $i_2$} if and only if there is a $(U,\alpha,\mathscr{e})$-incrementation for $\gamma$ with initial label $i_1$ and terminal label $i_2$.\end{definition}

The basic idea for an incrementation of a path $\gamma$ is to break it into segments $\gamma([t_j,t_{j+1}])$, and then label each such segment with a specific integer $k_j$ such that $\gamma([t_j,t_{j+1}])\subseteq\alpha^{k_j}(U)$. This labeling is further required to only move up or down by $1$ between adjacent segments, with the intersections occurring only in places which must be identified under the na\"{i}ve gluing. In other words, an incrementation amounts to describing a path within the quotient space $\bigsqcup_{i\in\mathbb{Z}}\tilde{\alpha}^i(U)/\sim_\text{naive}$ of the na\"ive gluing. We have attempted to illustrate the concept in Figures \ref{increment1pic} and \ref{increment2pic}.

\begin{figure}
\centering\includegraphics[width=0.6\textwidth]{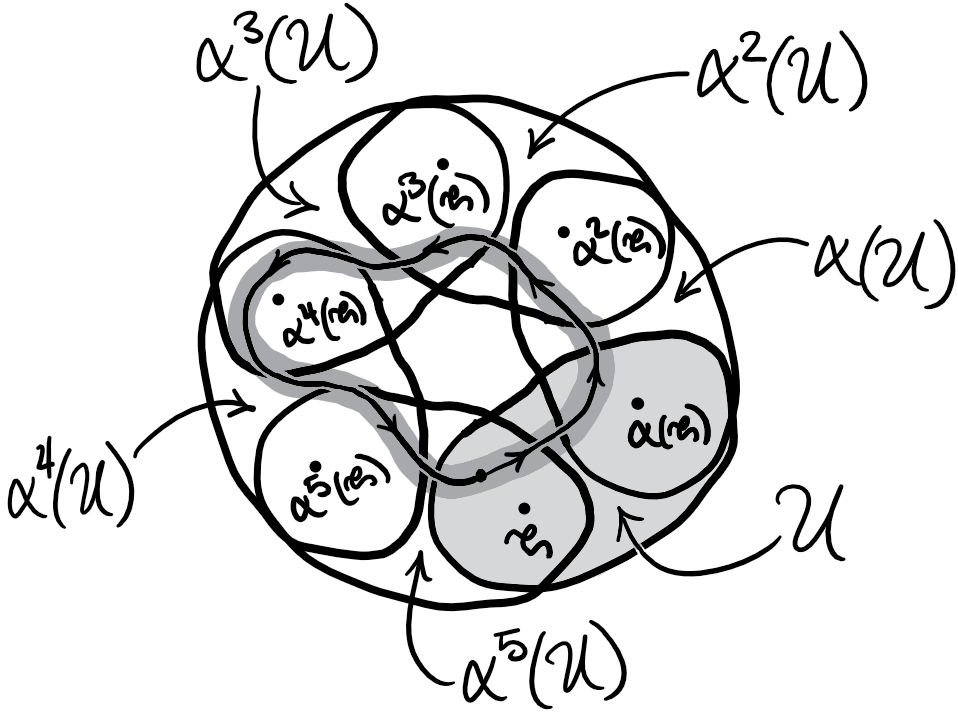}
\caption{A path $\gamma$, highlighted in darker gray, in the manifold $M$, as well the region $U$, highlighted in lighter gray, and its iterates under an automorphism $\alpha$}
\label{increment1pic}
\end{figure}

\begin{figure}
\centering\includegraphics[width=0.5\textwidth]{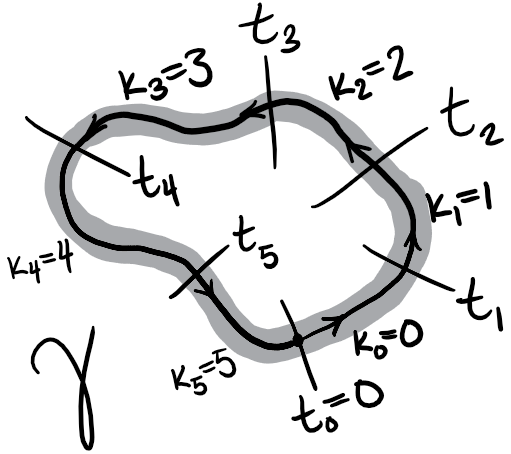}
\caption{An incrementation for the path $\gamma$ depicted in Figure \ref{increment1pic}}
\label{increment2pic}
\end{figure}

Recall that a null-homotopy based at a point $q_{{}_H}(\mathscr{g})\in M$ is a map $c:[0,1]^2\to M$, given as $(s,t)\mapsto c_s(t)$, such that \[c_s(0)=c_s(1)=c_1(s)=q_{{}_H}(\mathscr{g})\] for all $s\in[0,1]$. We will need to use a specific type of homotopy, called a \emph{thin} homotopy.

\begin{definition}A null-homotopy $c:[0,1]^2\to M$ is said to be \emph{thin} if and only if $c([0,1]^2)=c_0([0,1])$. Consequently, a loop $\gamma:[0,1]\to M$ is \emph{thinly null-homotopic} if and only if there exists a thin null-homotopy $c$ based at $\gamma(0)=\gamma(1)$ such that $c_0=\gamma$.\end{definition}

A thin null-homotopy from a loop $\gamma:[0,1]\to M$ to the constant loop at $\gamma(0)=\gamma(1)$ deforms $\gamma$ to a point while staying within its own image. The archetypical example of a thinly null-homotopic loop is the concatenation of a path with its reverse, so that the resulting loop ``backtracks'' over itself. Thin homotopies are geometrically useful in many contexts because thinly homotopic loops always have the same holonomy (see, for example, \cite{CaetanoPicken1994}). In particular, while we do not make explicit use of this outside of the appendix in the current version of the paper, it is worth noting that thinly null-homotopic loops always have trivial holonomy.

With incrementations and thin null-homotopies in hand, we can now define sprawl-equivalence.

\begin{definition}Two elements $\tilde{\alpha}^{i_1}(\mathscr{g}_1)$ and $\tilde{\alpha}^{i_2}(\mathscr{g}_2)$ of $\bigsqcup_{i\in\mathbb{Z}}\tilde{\alpha}^i(q_{{}_H}^{-1}(U))$ are said to be \emph{sprawl-equivalent}, denoted by $\tilde{\alpha}^{i_1}(\mathscr{g}_1)\sim\tilde{\alpha}^{i_2}(\mathscr{g}_2)$, if and only if $\alpha^{i_1}(\mathscr{g}_1)=\alpha^{i_2}(\mathscr{g}_2)$ and there exists a thinly null-homotopic loop $\gamma:[0,1]\to M$ that is based at \[\gamma(0)=q_{{}_H}(\alpha^{i_1}(\mathscr{g}_1))=q_{{}_H}(\alpha^{i_2}(\mathscr{g}_2))=\gamma(1)\] and incremented from $i_1$ to $i_2$.\end{definition}

\begin{proposition}Sprawl-equivalence is an equivalence relation.\end{proposition}
\begin{proof}We want to show that $\sim$ is reflexive, symmetric, and transitive.

For each $i\in\mathbb{Z}$ and $\mathscr{g}\in q_{{}_H}^{-1}(U)$, choosing $\gamma$ to be the constant path at $\alpha^i(q_{{}_H}(\mathscr{g}))$, our partition to be the trivial partition $0=t_0<t_1=1$, and $i=k_0=i$ shows us that $\tilde{\alpha}^i(\mathscr{g})\sim\tilde{\alpha}^i(\mathscr{g})$.

By definition, if $\tilde{\alpha}^{i_1}(\mathscr{g}_1)\sim\tilde{\alpha}^{i_2}(\mathscr{g}_2)$, then $\alpha^{i_1}(\mathscr{g}_1)=\alpha^{i_2}(\mathscr{g}_2)$ and there must be a thinly null-homotopic loop $\gamma:[0,1]\to M$ based at \[\gamma(0)=q_{{}_H}(\alpha^{i_1}(\mathscr{g}_1))=q_{{}_H}(\alpha^{i_2}(\mathscr{g}_2))=\gamma(1)\] with an incrementation given by a partition $0=t_0<\cdots<t_\ell=1$ and a sequence of integers $i_1=k_0,\dots,k_{\ell-1}=i_2\in\mathbb{Z}$. Consider the reverse loop $\bar{\gamma}:[0,1]\to M$ defined by $t\mapsto\gamma(1-t)$; setting $\bar{t}_j=1-t_{\ell-j}$ and $\bar{k}_j=k_{\ell-1-j}$ for each $j$, we then get a reversed incrementation from $i_2$ to $i_1$ for the thinly null-homotopic loop $\bar{\gamma}$, so $\tilde{\alpha}^{i_2}(\mathscr{g}_2)\sim\tilde{\alpha}^{i_1}(\mathscr{g}_1)$.

Similarly, if we have $\tilde{\alpha}^{i_1}(\mathscr{g}_1)\sim\tilde{\alpha}^{i_2}(\mathscr{g}_2)$ and $\tilde{\alpha}^{i_2}(\mathscr{g}_2)\sim\tilde{\alpha}^{i_3}(\mathscr{g}_3)$, then there exist corresponding thinly null-homotopic loops $\gamma$ and $\gamma'$ on $M$, together with incrementations given by $0=t_0<\cdots<t_\ell=1$ and $i_1=k_0,\dots,k_{\ell-1}=i_2\in\mathbb{Z}$ for $\gamma$, and $0=t'_0<\cdots<t'_{\ell'}=1$ and $i_2=k'_0,\dots,k'_{\ell'-1}=i_3\in\mathbb{Z}$ for $\gamma'$. To show that $\tilde{\alpha}^{i_1}(\mathscr{g}_1)\sim\tilde{\alpha}^{i_3}(\mathscr{g}_3)$, consider the concatenated loop $\gamma\star\gamma'$. This is still thinly null-homotopic, and setting \[\tau_j=\begin{cases}\frac{t_j}{2} & \text{if } j<\ell, \\ \frac{1+t'_{j-\ell}}{2} & \text{if } j\geq\ell\end{cases}\] and \[r_j=\begin{cases}k_j & \text{if } j<\ell, \\ k'_{j-\ell+1} & \text{if } j\geq\ell\end{cases}\] for each $j$, we get an incrementation for the concatenation $\gamma\star\gamma'$ comprised of the partition $0=\tau_0<\cdots<\tau_{\ell+\ell'}=1$ and labels $i_1=r_0,\dots,r_{\ell+\ell'-1}=i_3\in\mathbb{Z}$. In particular, $\tilde{\alpha}^{i_1}(\mathscr{g}_1)\sim\tilde{\alpha}^{i_3}(\mathscr{g}_3)$.\mbox{\qedhere}\end{proof}

Naturally, sprawl-equivalence induces a corresponding equivalence relation on the base manifold $\bigsqcup_{i\in\mathbb{Z}}\tilde{\alpha}^i(U)$, so that a point $\tilde{\alpha}^{i_1}(q_{{}_H}(\mathscr{g}_1))$ is identified with $\tilde{\alpha}^{i_2}(q_{{}_H}(\mathscr{g}_2))$ if and only if $\tilde{\alpha}^{i_1}(\mathscr{g}_1)\sim\tilde{\alpha}^{i_2}(\mathscr{g}_2h)$ for some $h\in H$. As before, we will refer to both of these equivalence relations as sprawl-equivalence, and denote them by the same symbol $\sim$.

Sprawl-equivalence is precisely the correction of the na\"{i}ve gluing that was mentioned at the start of the section; as we shall see shortly, it allows us to glue the copies $\tilde{\alpha}^i(q_{{}_H}^{-1}(U))$ together into a new principal $H$-bundle $\mathscr{F}$ such that $\tilde{\alpha}(\mathscr{e})$ coincides with $\alpha(\mathscr{e})$.

\begin{proposition}The quotient space \[\mathscr{F}=\mathscr{F}(q_{{}_H}^{-1}(U),\alpha,\mathscr{e}):=\left(\bigsqcup_{i\in\mathbb{Z}}\tilde{\alpha}^i(q_{{}_H}^{-1}(U))\right){\Big /}\sim\] admits the structure of a (smooth) principal $H$-bundle over the quotient space \[q_{{}_H}(\mathscr{F})=q_{{}_H}(\mathscr{F})(U,\alpha,\mathscr{e}):=\left(\bigsqcup_{i\in\mathbb{Z}}\tilde{\alpha}^i(U)\right){\Big /}\sim~,\] which is a smooth manifold.\end{proposition}
\begin{proof}For each $i\in\mathbb{Z}$, $\tilde{\alpha}^i(\mathscr{g}_1)\sim\tilde{\alpha}^i(\mathscr{g}_2)$ if and only if $\mathscr{g}_1=\mathscr{g}_2$.
This means that the quotient map by $\sim$ is injective when restricted to each $\tilde{\alpha}^i(q_{{}_H}^{-1}(U))$. Moreover, open subsets $V\subseteq\bigsqcup_{i\in\mathbb{Z}}\tilde{\alpha}^i(q_{{}_H}^{-1}(U))$ are mapped to open subsets of the quotient: if $\tilde{\alpha}^{i_1}(\mathscr{g}_1)\in V$ and $\tilde{\alpha}^{i_1}(\mathscr{g}_1)\sim\tilde{\alpha}^{i_2}(\mathscr{g}_2)$, with $\gamma$ the corresponding thinly null-homotopic loop incremented from $i_1$ to $i_2$, then for each small path $\delta$ in $\alpha^{i_1}(q_{{}_H}^{-1}(U))\cap\alpha^{i_2}(q_{{}_H}^{-1}(U))$ such that $\delta(0)=\alpha^{i_1}(\mathscr{g}_1)=\alpha^{i_2}(\mathscr{g}_2)$ and $\tilde{\alpha}^{i_1}(\alpha^{-i_1}(\delta(1)))\in V$, we can take the concatenation $\overline{q_{{}_H}(\delta)}\star\gamma\star q_{{}_H}(\delta)$ of the reverse of $q_{{}_H}(\delta)$ with $\gamma$ with $q_{{}_H}(\delta)$ to get a thinly null-homotopic loop incremented from $i_1$ to $i_2$, which tells us that $\tilde{\alpha}^{i_2}(\alpha^{-i_2}(\delta(1)))$ is sprawl-equivalent to an element $\tilde{\alpha}^{i_1}(\alpha^{-i_1}(\delta(1)))\in V$, hence the union of sprawl-equivalence classes of elements of $V$ is open. Thus, the quotient map from $\bigsqcup_{i\in\mathbb{Z}}\tilde{\alpha}^i(q_{{}_H}^{-1}(U))$ to $\mathscr{F}$ restricts to an embedding on each $\tilde{\alpha}^i(q_{{}_H}^{-1}(U))$, so it makes sense to identify each $\tilde{\alpha}^i(q_{{}_H}^{-1}(U))$ with its image in $\mathscr{F}$. Similarly, each $\tilde{\alpha}^i(U)$ naturally embeds into $q_{{}_H}(\mathscr{F})$, so we can identify each $\tilde{\alpha}^i(U)$ with its image in the quotient space $q_{{}_H}(\mathscr{F})$.


For every element $h\in H$, we have $\tilde{\alpha}^{i_1}(\mathscr{g}_1)\sim\tilde{\alpha}^{i_2}(\mathscr{g}_2)$ if and only if $\tilde{\alpha}^{i_1}(\mathscr{g}_1)h\sim\tilde{\alpha}^{i_2}(\mathscr{g}_2)h$, and $\tilde{\alpha}^i(\mathscr{g})\sim\tilde{\alpha}^i(\mathscr{g})h$ if and only if $h$ is the identity element because otherwise $\alpha^i(\mathscr{g})\neq\alpha^i(\mathscr{g})h$. Because of this, $\mathscr{F}$ inherits a free right $H$-action that coincides with the smooth free right action of $H$ on each $\tilde{\alpha}^i(q_{{}_H}^{-1}(U))$. Since $\tilde{\alpha}^{i_1}(\mathscr{g}_1)\sim\tilde{\alpha}^{i_2}(\mathscr{g}_2)$ implies $\tilde{\alpha}^{i_1}(q_{{}_H}(\mathscr{g}_1))\sim\tilde{\alpha}^{i_2}(q_{{}_H}(\mathscr{g}_2))$, we get a natural map $q_{{}_H}:\mathscr{F}\to q_{{}_H}(\mathscr{F})$ given by $q_{{}_H}(\tilde{\alpha}^i(\mathscr{g})):=\tilde{\alpha}^i(q_{{}_H}(\mathscr{g}))$. By definition, this coincides with the bundle map $q_{{}_H}:\tilde{\alpha}^i(q_{{}_H}^{-1}(U))\to\tilde{\alpha}^i(U)$ for each $i$, so $\mathscr{F}$ is a principal $H$-bundle over $q_{{}_H}(\mathscr{F})$.

It remains to show that $q_{{}_H}(\mathscr{F})$ is a smooth manifold. Note that $U$ naturally inherits a smooth structure from $M$, and $q_{{}_H}(\mathscr{F})$ is a union of embedded copies of $U$ by definition. Moreover, $\tilde{\alpha}^{i_1}(\mathscr{g}_1)\sim\tilde{\alpha}^{i_2}(\mathscr{g}_2)$ implies $\alpha^{i_1}(\mathscr{g}_1)=\alpha^{i_2}(\mathscr{g}_2)$, hence $\mathscr{g}_2=\alpha^{i_1-i_2}(\mathscr{g}_1)$, so the embedded copies of $U$ are glued together in $q_{{}_H}(\mathscr{F})$ along open sets by iterates of the diffeomorphism $\alpha$. In particular, we just need to show that $q_{{}_H}(\mathscr{F})$ is Hausdorff to verify that it admits the structure of a smooth manifold.

To this end, suppose that $\tilde{\alpha}^{i_1}(q_{{}_H}(\mathscr{g}_1))$ and $\tilde{\alpha}^{i_2}(q_{{}_H}(\mathscr{g}_2))$ are distinct points of the quotient space $q_{{}_H}(\mathscr{F})$. There are two possible cases: either $\alpha^{i_1}(q_{{}_H}(\mathscr{g}_1))\neq\alpha^{i_2}(q_{{}_H}(\mathscr{g}_2))$, or $\alpha^{i_1}(q_{{}_H}(\mathscr{g}_1))=\alpha^{i_2}(q_{{}_H}(\mathscr{g}_2))$ but there is no corresponding thinly null-homotopic loop incremented from $i_1$ to $i_2$. In the first case, there exist disjoint open neighborhoods $V_1\subseteq\alpha^{i_1}(U)$ of $\alpha^{i_1}(q_{{}_H}(\mathscr{g}_1))$ and $V_2\subseteq\alpha^{i_2}(U)$ of $\alpha^{i_2}(q_{{}_H}(\mathscr{g}_2))$ because $M$ is Hausdorff, hence $\tilde{\alpha}^{i_1}(\alpha^{-i_1}(V_1))$ and $\tilde{\alpha}^{i_2}(\alpha^{-i_2}(V_2))$ are disjoint open neighborhoods of $\tilde{\alpha}^{i_1}(q_{{}_H}(\mathscr{g}_1))$ and $\tilde{\alpha}^{i_2}(q_{{}_H}(\mathscr{g}_2))$, respectively. In the second case, let $V$ be the path component of the point $\alpha^{i_1}(q_{{}_H}(\mathscr{g}_1))=\alpha^{i_2}(q_{{}_H}(\mathscr{g}_2))$ in the intersection $\alpha^{i_1}(U)\cap\alpha^{i_2}(U)$, so that $\tilde{\alpha}^{i_1}(\alpha^{-i_1}(V))$ is an open neighborhood of $\tilde{\alpha}^{i_1}(q_{{}_H}(\mathscr{g}_1))$ and $\tilde{\alpha}^{i_2}(\alpha^{-i_2}(V))$ is an open neighborhood of $\tilde{\alpha}^{i_2}(q_{{}_H}(\mathscr{g}_2))$. These two neighborhoods must be disjoint: if there were a point $q_{{}_H}(\mathscr{f})$ in their intersection, then there would be a path \[\zeta:[0,1]\to V\subseteq\alpha^{i_1}(U)\cap\alpha^{i_2}(U)\] from $\alpha^{i_1}(q_{{}_H}(\mathscr{g}_1))=(\alpha^{i_1}\circ(\tilde{\alpha}^{i_1})^{-1})(\tilde{\alpha}^{i_1}(q_{{}_H}(\mathscr{g}_1)))$ to $(\alpha^{i_1}\circ(\tilde{\alpha}^{i_1})^{-1})(q_{{}_H}(\mathscr{f}))$, so if $\gamma$ were the thinly null-homotopic loop based at \[(\alpha^{i_1}\circ(\tilde{\alpha}^{i_1})^{-1})(q_{{}_H}(\mathscr{f}))=(\alpha^{i_2}\circ(\tilde{\alpha}^{i_2})^{-1})(q_{{}_H}(\mathscr{f}))\] incremented from $i_1$ to $i_2$ that must exist for the point $q_{{}_H}(\mathscr{f})$ to be in the intersection of $\tilde{\alpha}^{i_1}(U)$ and $\tilde{\alpha}^{i_2}(U)$, then the concatenation given by $\zeta\star\gamma\star\bar{\zeta}$ would be a thinly null-homotopic loop based at the point $\alpha^{i_1}(q_{{}_H}(\mathscr{g}_1))=\alpha^{i_2}(q_{{}_H}(\mathscr{g}_2))$ and incremented from $i_1$ to $i_2$. This would be a contradiction, since $\tilde{\alpha}^{i_1}(q_{{}_H}(\mathscr{g}_1))$ and $\tilde{\alpha}^{i_2}(q_{{}_H}(\mathscr{g}_2))$ are distinct by assumption, so $\tilde{\alpha}^{i_1}(\alpha^{-i_1}(V))$ and $\tilde{\alpha}^{i_2}(\alpha^{-i_2}(V))$ must be disjoint. Thus, $q_{{}_H}(\mathscr{F})$ is Hausdorff.\mbox{\qedhere}\end{proof}

To imbue this new principal $H$-bundle with the structure of a Cartan geometry, we will use a natural map from $\mathscr{F}$ to $\mathscr{G}$ in order to pull the Cartan connection on $\mathscr{G}$ back to $\mathscr{F}$. This map, called the \emph{sprawl map}, is precisely the one obtained by identifying each $\tilde{\alpha}^i(q_{{}_H}^{-1}(U))$ embedded in $\mathscr{F}$ with the corresponding $\alpha^i(q_{{}_H}^{-1}(U))$ in $\mathscr{G}$.

\begin{definition}The map $\sigma:\mathscr{F}\to\mathscr{G}$ given by $\tilde{\alpha}^i(\mathscr{g})\mapsto\alpha^i(\mathscr{g})$ is called the \emph{sprawl map} for $(\mathscr{G},\omega)$.\end{definition}

Before moving on to defining the sprawl, let us make two observations about the sprawl map. First, $\sigma$ is well-defined: $\tilde{\alpha}^{i_1}(\mathscr{g}_1)\sim\tilde{\alpha}^{i_2}(\mathscr{g}_2)$ only if \[\sigma(\tilde{\alpha}^{i_1}(\mathscr{g}_1))=\alpha^{i_1}(\mathscr{g}_1)=\alpha^{i_2}(\mathscr{g}_2)=\sigma(\tilde{\alpha}^{i_2}(\mathscr{g}_2)),\] so sprawl-equivalent elements have the same image under $\sigma$. Second, $\sigma$ is an $H$-equivariant local diffeomorphism, since it coincides with the natural $H$-equivariant diffeomorphism between $\tilde{\alpha}^i(q_{{}_H}^{-1}(U))$ and $\alpha^i(q_{{}_H}^{-1}(U))$ for each $i\in\mathbb{Z}$.

With that, we can finally define the sprawl.

\begin{definition}\label{sprawldef} The \emph{sprawl of $(q_{{}_H}^{-1}(U),\omega)$ generated by $\alpha$ from $\mathscr{e}$} is the Cartan geometry $(\mathscr{F},\sigma^*\omega)$ of type $(G,H)$ over $q_{{}_H}(\mathscr{F})$, where $\sigma$ is the sprawl map.\end{definition}

Crucially, note that we have constructed $(\mathscr{F},\sigma^*\omega)$ in such a way as to make the map \[\tilde{\alpha}:\mathscr{F}\to\mathscr{F},\,\tilde{\alpha}^i(\mathscr{g})\mapsto\tilde{\alpha}^{i+1}(\mathscr{g})\] into an automorphism. Indeed, $\sigma$ naturally satisfies $\sigma\circ\tilde{\alpha}=\alpha\circ\sigma$, so \[\tilde{\alpha}^*(\sigma^*\omega)=(\sigma\circ\tilde{\alpha})^*\omega=(\alpha\circ\sigma)^*\omega=\sigma^*(\alpha^*\omega)=\sigma^*\omega.\] Moreover, $\tilde{\alpha}$ and $\alpha$ must coincide on the distinguished element $\mathscr{e}$ under the identification between $\tilde{\alpha}^0(q_{{}_H}^{-1}(U))$ and $q_{{}_H}^{-1}(U)$, so $\tilde{\alpha}$ has the same local behavior as $\alpha$ on $q_{{}H}^{-1}(U)$ near $\mathscr{e}$.


\subsection{The universal property of sprawls}\label{sprawl3}
We would like to think of the automorphism $\tilde{\alpha}$ on the sprawl $(\mathscr{F},\sigma^*\omega)$ as a kind of universal example of an automorphism with the same behavior as $\alpha$ near $\mathscr{e}$. Theorem \ref{universalproperty} will make precise what we mean by ``universal example'', but first, we will need two lemmas.

First, we need to show that lifts of incremented paths to $\mathscr{G}$ further lift to paths on $\mathscr{F}$ via the sprawl map, and that the choice of lift only depends on the initial label of the underlying incrementation.

\begin{lemma}\label{incrlift} If $\gamma:[0,1]\to\mathscr{G}$ is a path in $\mathscr{G}$ such that its image $q_{{}_H}(\gamma)$ in $M$ has an incrementation, then there exists a lift $\tilde{\gamma}:[0,1]\to\mathscr{F}$ of $\gamma$, so that $\sigma\circ\tilde{\gamma}=\gamma$. Moreover, this choice of lift only depends on the initial label of the incrementation of $q_{{}_H}(\gamma)$.\end{lemma}
\begin{proof}Suppose that the incrementation of $q_{{}_H}(\gamma)$ is the one given by the partition $0=t_0<\cdots<t_\ell=1$ and labels $k_0,\dots,k_{\ell-1}\in\mathbb{Z}$. We can construct a path $\tilde{\gamma}$ in $\mathscr{F}$ as follows. First, let us direct our attention to $\alpha^{k_0}(q_{{}_H}^{-1}(U))$, where the path $\gamma$ starts. When restricted to $\tilde{\alpha}^{k_0}(q_{{}_H}^{-1}(U))$, $\sigma$ coincides with the identification between $\tilde{\alpha}^{k_0}(q_{{}_H}^{-1}(U))$ and $\alpha^{k_0}(q_{{}_H}^{-1}(U))$, so we can simply define \[\tilde{\gamma}|_{[0,t_1]}:=(\sigma|_{\tilde{\alpha}^{k_0}(q_{{}_H}^{-1}(U))})^{-1}\circ\gamma|_{[0,t_1]}.\] Next, the incrementation tells us that $q_{{}_H}(\gamma(t_1))$ is in the connected component of $\alpha^{\max(k_0,k_1)}(q_{{}_H}(\mathscr{e}))$ in $\alpha^{k_0}(U)\cap\alpha^{k_1}(U)$, so that the constant path at $q_{{}_H}(\gamma(t_1))$ is a thinly null-homotopic loop incremented from $k_0$ to $k_1$. In particular, this tells us that $\tilde{\gamma}(t_1)\in\tilde{\alpha}^{k_0}(q_{{}_H}^{-1}(U))\cap\tilde{\alpha}^{k_1}(q_{{}_H}^{-1}(U))$, so that we can extend the path $\tilde{\gamma}$ by again restricting to where $\sigma$ is a diffeomorphism: \[\tilde{\gamma}|_{[t_1,t_2]}:=(\sigma|_{\tilde{\alpha}^{k_1}(q_{{}_H}^{-1}(U))})^{-1}\circ\gamma|_{[t_1,t_2]}.\] By iterating this procedure, defining \[\tilde{\gamma}|_{[t_j,t_{j+1}]}:=(\sigma|_{\tilde{\alpha}^{k_j}(q_{{}_H}^{-1}(U))})^{-1}\circ\gamma|_{[t_j,t_{j+1}]}\] for each $j$, we get a well-defined lift $\tilde{\gamma}$ of $\gamma$ to $\mathscr{F}$, with $\sigma\circ\tilde{\gamma}=\gamma$.

Now, suppose $\zeta:[0,1]\to\mathscr{F}$ is another lift of $\gamma$ to $\mathscr{F}$, constructed in the same way from a possibly different incrementation of $q_{{}_H}(\gamma)$. Then, by definition, we would again have $\sigma\circ\zeta=\gamma$, and since $\sigma$ is a geometric map, this means that $\zeta$, $\tilde{\gamma}$, and $\gamma$ would all have the same development: $\zeta_G=\tilde{\gamma}_G=\gamma_G$. In particular, since the starting points of $\zeta$ and $\tilde{\gamma}$ are uniquely determined by the initial label for $q_{{}_H}(\gamma)$, we must have $\zeta=\tilde{\gamma}$ if their corresponding incrementations have the same initial label, since then they have the same starting point and the same developments.\mbox{\qedhere}\end{proof}

Our second lemma shows us that development completely determines when a path is a thinly null-homotopic loop.

\begin{lemma}\label{backtracking} A path $\gamma:[0,1]\to\mathscr{G}$ is a thinly null-homotopic loop if and only if its development $\gamma_G:[0,1]\to G$ is.\end{lemma}
\begin{proof}Suppose $\gamma:[0,1]\to\mathscr{G}$ is a thinly null-homotopic loop in $\mathscr{G}$. Up to smooth reparametrization, we may assume that both $\gamma$ and the thin null-homotopy $c:[0,1]^2\to\mathscr{G}$ are smooth. Because $c$ is thin, the image of $c$ is at most one-dimensional, so $c^*\omega$ satisfies $\mathrm{d}(c^*\omega)+\frac{1}{2}[c^*\omega,c^*\omega]=0$. By the fundamental theorem of nonabelian calculus (Theorem 7.14 in Chapter 3 of \cite{Sharpe1997}), it follows that there is a unique smooth map $c_G:[0,1]^2\to G$ such that both $(c_G)_0(0)=e$ and $c_G^*\MC{G}=c^*\omega$; because $\gamma=c_0$ and $\gamma_G(0)=e=(c_G)_0(0)$, this must also satisfy $(c_G)_0=\gamma_G$. Since $c_G([0,1]^2)=(c_G)_0([0,1])$ and $c_G$ is constant along $[0,1]\times\{0\}$, $\{1\}\times [0,1]$, and $[0,1]\times\{1\}$, we see that $c_G$ is a thin null-homotopy from $\gamma_G$ to the constant path at $e$.

Conversely, suppose $\gamma_G$ is a thinly null-homotopic loop. Again, up to smooth reparametrization, we may assume that both $\gamma$ and the thin null-homotopy $c_G:[0,1]^2\to G$ with $(c_G)_0=\gamma_G$ are smooth. Our strategy is essentially to just modify the local version of the fundamental theorem of nonabelian calculus to show that a map $c:[0,1]^2\to\mathscr{G}$ with $c^*\omega=c_G^*\MC{G}$ exists locally, then build the map from these local pieces starting at $c_0(0)=\gamma(0)$. Since such a map $c$ must be constant along $[0,1]\times\{0\}$, $\{1\}\times [0,1]$, and $[0,1]\times\{1\}$, and $(c_0)_G=(c_G)_0=\gamma_G$, it will be a null-homotopy from $\gamma$ to $\gamma(0)$ if it exists, and the image of $c$ cannot leave the image of $c_0=\gamma$ because the image of $c_G$ is contained in the image of $\gamma_G$, so $c$ is necessarily a thin null-homotopy.

Emulating the proof of Theorem 6.1 in Chapter 3 of \cite{Sharpe1997}, we consider the projections $\pi_\mathscr{G}:[0,1]^2\times\mathscr{G}\to\mathscr{G}$ and $\pi_{[0,1]^2}:[0,1]^2\times\mathscr{G}\to[0,1]^2$. Setting $\zeta:=(c_G\circ\pi_{[0,1]^2})^*\MC{G}-\pi_\mathscr{G}^*\omega$, we see that $\pi_{[0,1]^2\,*}$ gives a linear isomorphism from $\ker(\zeta)$ to the tangent spaces of $[0,1]^2$, so that $\ker(\zeta)$ is a two-dimensional distribution. Moreover, for $\Omega:=\mathrm{d}\omega+\frac{1}{2}[\omega,\omega]$, \begin{align*}\mathrm{d}\zeta & =(c_G\circ\pi_{[0,1]^2})^*\mathrm{d}\MC{G}-\pi_\mathscr{G}^*\mathrm{d}\omega \\ & =-\frac{1}{2}(c_G\circ\pi_{[0,1]^2})^*[\MC{G},\MC{G}]+\frac{1}{2}\pi_\mathscr{G}^*[\omega,\omega]-\pi_\mathscr{G}^*\Omega \\ & =-\frac{1}{2}[\zeta+\pi_\mathscr{G}^*\omega,\zeta+\pi_\mathscr{G}^*\omega]+\frac{1}{2}\pi_\mathscr{G}^*[\omega,\omega]-\pi_\mathscr{G}^*\Omega \\ & =-\frac{1}{2}([\zeta,\zeta]+[\pi_\mathscr{G}^*\omega,\zeta]+[\zeta,\pi_\mathscr{G}^*\omega])-\pi_\mathscr{G}^*\Omega.\end{align*} Since $c_G$ has rank at most one, $(\pi_\mathscr{G})_*\ker(\zeta)$ is at most one-dimensional, so $\pi_\mathscr{G}^*\Omega$ must vanish on $\ker(\zeta)$. The rest of the expression for $\mathrm{d}\zeta$ above is a sum of terms formed by bracketing with $\zeta$, so it must vanish on $\ker(\zeta)$ as well. Thus, $\ker(\zeta)$ is integrable. If $N$ is a leaf of $\ker(\zeta)$ through $((s,t),\mathscr{g})\in[0,1]^2\times\mathscr{G}$, then $\pi_{[0,1]^2\,*}$ gives a linear isomorphism from the tangent space of $N$ at $((s,t),\mathscr{g})$ to the tangent space of $[0,1]^2$ at $(s,t)$, so there is a neighborhood $V$ of $(s,t)$ on which we get a smooth inverse $f:V\to N$ to $\pi_{[0,1]^2}$ such that $f(s,t)=((s,t),\mathscr{g})$. Thus, \begin{align*}0=f^*\zeta & =f^*((c_G\circ\pi_{[0,1]^2})^*\MC{G}-\pi_\mathscr{G}^*\omega) \\ & =f^*\pi_{[0,1]^2}^*(c_G^*\MC{G})-f^*\pi_\mathscr{G}^*\omega \\ & =(\pi_{[0,1]^2}\circ f)^*(c_G^*\MC{G})-(\pi_\mathscr{G}\circ f)^*\omega \\ & =c_G^*\MC{G}-(\pi_\mathscr{G}\circ f)^*\omega,\end{align*} so $c|_V:=\pi_\mathscr{G}\circ f:V\subseteq[0,1]^2\to\mathscr{G}$ satisfies $(c|_V)^*\omega=c_G^*\MC{G}$.

For each $s\in[0,1]$, let $c_s:[0,1]\to\mathscr{G}$ be the unique path with $c_s(0)=\gamma(0)$ and $(c_s)_G=(c_G)_s$; since $c_0=\gamma$ is well-defined and each $c_s$ must stay within the image of $c_0$, these paths are well-defined as well. If there is a map $c:[0,1]^2\to\mathscr{G}$ with $c^*\omega=c_G^*\MC{G}$ and $c_0(0)=\gamma(0)$, then it must satisfy $c|_{\{s\}\times[0,1]}=c_s$ for each $s$, so we just need to verify that $c:(s,t)\mapsto c_s(t)$ works as our map. To do this, choose an open neighborhood $V_{(s,t)}$ for each $(s,t)\in[0,1]^2$ such that we get a map $c|_{V_{(s,t)}}$ as above with $(c|_{V_{(s,t)}})(s,t):=c_s(t)$ and $(c|_{V_{(s,t)}})^*\omega=c_G^*\MC{G}$. This lets us cover each $\{s\}\times[0,1]$ with open sets on which a map satisfying the desired conditions exists, and these maps $c|_{V_{(s,t)}}$ would necessarily agree on overlaps along $\{s\}\times[0,1]$ because, by definition, $c_s$ is the unique path with $c_s(0)=\gamma(0)$ and $(c_s)_G=(c_G)_s$. Thus, for each $s\in[0,1]$, setting $V_s:=\bigcup_{t\in[0,1]}V_{(s,t)}$, we get a map $c|_{V_s}$ on an open neighborhood of $\{s\}\times[0,1]$ such that $(c|_{V_s})^*\omega=c_G^*\MC{G}$ and $(c|_{V_s})(s,0)=\gamma(0)$. From here, we can glue the maps $c|_{V_s}$ together along their overlaps to get $c$, since the $c|_{V_s}$ must necessarily coincide on $[0,1]\times\{0\}$ because they are constant along this interval. Thus, we get a map $c:[0,1]^2\to\mathscr{G}$ satisfying $c^*\omega=c_G^*\MC{G}$ and $c_0(0)=\gamma(0)$, which must be a thin null-homotopy by the argument above.\mbox{\qedhere}\end{proof}

With these lemmas in hand, let us finally explain what the following theorem is meant to tell us. Recall that, in Definition \ref{sprawldef}, we refer to the Cartan geometry $(\mathscr{F},\sigma^*\omega)$ as ``the sprawl of $(q_{{}_H}^{-1}(U),\omega)$ generated by $\alpha$ from $\mathscr{e}$''. Ostensibly, however, $(q_{{}_H}^{-1}(U),\omega)$, $\alpha$, and $\mathscr{e}$ are not enough to determine the geometric structure of the sprawl: the Cartan connection is given explicitly in terms of the sprawl map $\sigma$ for $(\mathscr{G},\omega)$, and the topology of $\mathscr{F}$ is determined by particular null-homotopies in $M$. We would like to show that, in truth, the sprawl really is uniquely determined by $(q_{{}_H}^{-1}(U),\omega)$, the distinguished element $\mathscr{e}\in q_{{}_H}^{-1}(U)$, and the behavior of $\alpha$ on them.

To do this, suppose $(\mathscr{Q},\upsilon)$ is another Cartan geometry of type $(G,H)$ that happens to have an open set geometrically identical to $(q_{{}_H}^{-1}(U),\omega)$, meaning that there is a geometric embedding $\psi:(q_{{}_H}^{-1}(U),\omega)\hookrightarrow (\mathscr{Q},\upsilon)$. Furthermore, suppose it has an automorphism $\varphi\in\Aut(\mathscr{Q},\upsilon)$ that behaves exactly as $\alpha$ does on the distinguished element $\mathscr{e}$ under the identification given by the geometric embedding $\psi$; in other words, $\varphi(\psi(\mathscr{e}))=\psi(\alpha(\mathscr{e}))$. Then, if the sprawl truly is uniquely determined by $(q_{{}_H}^{-1}(U),\omega)$, $\alpha$, and $\mathscr{e}$, then the sprawl of $(\psi(q_{{}_H}^{-1}(U)),\upsilon)$ generated by $\varphi$ from $\psi(\mathscr{e})$ should be geometrically isomorphic to $(\mathscr{F},\sigma^*\omega)$ in some natural way. The following theorem shows exactly this; indeed, it shows that the embedding $\psi$ uniquely extends to the new sprawl map for $(\mathscr{Q},\upsilon)$ from $(\mathscr{F},\sigma^*\omega)$.

\begin{theorem}\label{universalproperty} Let $(\mathscr{Q},\upsilon)$ be another Cartan geometry of type $(G,H)$, with an automorphism $\varphi\in\Aut(\mathscr{Q},\upsilon)$. If \[\psi|_{q_{{}_H}^{-1}(U)}:(\tilde{\alpha}^0(q_{{}_H}^{-1}(U)),\sigma^*\omega)\cong(q_{{}_H}^{-1}(U),\omega)\hookrightarrow (\mathscr{Q},\upsilon)\] is a geometric embedding such that $\varphi((\psi|_{q_{{}_H}^{-1}(U)})(\mathscr{e}))=(\psi|_{q_{{}_H}^{-1}(U)})(\alpha(\mathscr{e}))$, then $\psi|_{q_{{}_H}^{-1}(U)}$ has a unique extension to a geometric map \[\psi:(\mathscr{F},\sigma^*\omega)\to(\mathscr{Q},\upsilon)\] from the sprawl of $(q_{{}_H}^{-1}(U),\omega)$ generated by $\alpha$ from $\mathscr{e}$ into $(\mathscr{Q},\upsilon)$ such that $\psi\circ\tilde{\alpha}=\varphi\circ\psi$.\end{theorem}
\begin{proof}If the desired extension to the sprawl $\mathscr{F}$ exists, then it must be of the form $\psi:\tilde{\alpha}^i(\mathscr{g})\mapsto\varphi^i((\psi|_{q_{{}_H}^{-1}(U)})(\mathscr{g}))$, so uniqueness is immediate and \begin{align*}(\psi^*\upsilon)_{\tilde{\alpha}^i(\mathscr{g})} & =\psi^*(\upsilon_{\varphi^i(\psi(\mathscr{g}))})=\psi^*(\varphi^{-i})^*(\upsilon_{\psi(\mathscr{g})})=(\varphi^{-i}\circ\psi)^*(\upsilon_{\psi(\mathscr{g})}) \\ & =(\psi\circ\tilde{\alpha}^{-i})^*(\upsilon_{\psi(\mathscr{g})})=(\tilde{\alpha}^{-i})^*\psi^*(\upsilon_{\psi(\mathscr{g})})=(\tilde{\alpha}^{-i})^*(\sigma^*\omega_\mathscr{g}) \\ & =(\sigma^*\omega)_{\tilde{\alpha}^i(\mathscr{g})},\end{align*} hence $\psi$ must be a geometric map as well. Thus, we just need to show that an extension of this form is well-defined.

To this end, suppose $\tilde{\alpha}^{i_1}(\mathscr{g}_1)\sim\tilde{\alpha}^{i_2}(\mathscr{g}_2)$, so that $\alpha^{i_1}(\mathscr{g}_1)=\alpha^{i_2}(\mathscr{g}_2)$ and there exists a thinly null-homotopic loop $q_{{}_H}(\gamma):[0,1]\to M$ based at the point $\alpha^{i_1}(q_{{}_H}(\mathscr{g}_1))=\alpha^{i_2}(q_{{}_H}(\mathscr{g}_2))$ incremented from $i_1$ to $i_2$. Since the image of a null-homotopy is contractible, we can lift $q_{{}_H}(\gamma)$ to a thinly null-homotopic loop $\gamma:[0,1]\to\mathscr{G}$ based at $\alpha^{i_1}(\mathscr{g}_1)=\alpha^{i_2}(\mathscr{g}_2)$, and by Lemma \ref{incrlift}, we can further lift to a path $\tilde{\gamma}:[0,1]\to\mathscr{F}$ starting at $\tilde{\alpha}^{i_1}(\mathscr{g}_1)$. Since $\gamma_G=\tilde{\gamma}_G$, $\tilde{\gamma}$ is again a thinly null-homotopic loop by Lemma \ref{backtracking}. Our strategy to show that \[\psi(\tilde{\alpha}^{i_1}(\mathscr{g}_1))=\psi(\tilde{\gamma}(0))=\psi(\tilde{\gamma}(1))=\psi(\tilde{\alpha}^{i_2}(\mathscr{g}_2))\] is to construct a well-defined path $\beta:[0,1]\to\mathscr{Q}$ that always agrees with what the composite $\psi\circ\tilde{\gamma}$ is if $\psi$ is well-defined; because we will have $\beta_G=\tilde{\gamma}_G$, $\beta$ will be a thinly null-homotopic loop by Lemma \ref{backtracking}, hence \[\psi(\tilde{\gamma}(0))=\beta(0)=\beta(1)=\psi(\tilde{\gamma}(1)).\]

We construct the path $\beta:[0,1]\to\mathscr{Q}$ along the lines of the proof of Lemma \ref{incrlift}. Let the incrementation of $q_{{}_H}(\gamma)$ be given by the partition $0=t_0<\cdots<t_\ell=1$ and labels $i_1=k_0,\dots,k_{\ell-1}=i_2\in\mathbb{Z}$. To start, this means that $\tilde{\gamma}([0,t_1])\subseteq\tilde{\alpha}^{i_1}(q_{{}_H}^{-1}(U))$, since $\sigma(\tilde{\gamma})=\gamma$ by definition. Whenever we restrict $\psi$ to a given $\tilde{\alpha}^k(q_{{}_H}^{-1}(U))$, we get a well-defined geometric embedding $\psi|_{\tilde{\alpha}^k(q_{{}_H}^{-1}(U))}$, which by definition is given by \[\psi|_{\tilde{\alpha}^k(q_{{}_H}^{-1}(U))}:=\varphi^k\circ(\psi|_{q_{{}_H}^{-1}(U)})\circ\tilde{\alpha}^{-k}|_{\tilde{\alpha}^k(q_{{}_H}^{-1}(U))}.\] Therefore, it is valid to define $\beta|_{[0,t_1]}:=\psi|_{\tilde{\alpha}^{i_1}(q_{{}_H}^{-1}(U))}\circ\tilde{\gamma}|_{[0,t_1]}$.

At this point, we make a key observation: because the elements $\tilde{\alpha}(\mathscr{e})$ and $\alpha(\mathscr{e})$ are identified in $(\tilde{\alpha}^0(q_{{}_H}^{-1}(U)),\sigma^*\omega)\cong(q_{{}_H}^{-1}(U),\omega)$ and $(\psi|_{q_{{}_H}^{-1}(U)})(\tilde{\alpha}(\mathscr{e}))=\varphi((\psi|_{q_{{}_H}^{-1}(U)})(\mathscr{e}))$, the geometric embeddings $\psi|_{q_{{}_H}^{-1}(U)}$ and \[\psi|_{\tilde{\alpha}(q_{{}_H}^{-1}(U))}=\varphi\circ(\psi|_{q_{{}_H}^{-1}(U)})\circ\tilde{\alpha}^{-1}|_{\tilde{\alpha}(q_{{}_H}^{-1}(U))}\] must coincide over the connected component of $q_{{}_H}(\tilde{\alpha}(\mathscr{e}))=q_{{}_H}(\alpha(\mathscr{e}))$ in the intersection $U\cap\tilde{\alpha}(U)$ in $q_{{}_H}(\mathscr{F})$, since \[(\psi|_{\tilde{\alpha}(q_{{}_H}^{-1}(U))})(\tilde{\alpha}(\mathscr{e}))=\varphi((\psi|_{q_{{}_H}^{-1}(U)})(\mathscr{e}))=(\psi|_{q_{{}_H}^{-1}(U)})(\tilde{\alpha}(\mathscr{e})).\] Using iterates of $\tilde{\alpha}$ and $\varphi$ to move to the other copies of $q_{{}_H}^{-1}(U)$, we then see that, for each $k$, $\psi|_{\tilde{\alpha}^k(q_{{}_H}^{-1}(U))}$ and $\psi|_{\tilde{\alpha}^{k+1}(q_{{}_H}^{-1}(U))}$ must coincide over the connected component of $q_{{}_H}(\tilde{\alpha}^{k+1}(\mathscr{e}))$ in $\tilde{\alpha}^k(U)\cap\tilde{\alpha}^{k+1}(U)$. By definition, the incrementation of $q_{{}_H}(\gamma)$ tells us that $q_{{}_H}(\gamma)(t_1)$ lies in the connected component of the point $\alpha^{\max(k_0,k_1)}(q_{{}_H}(\mathscr{e}))$ in $\alpha^{k_0}(U)\cap\alpha^{k_1}(U)$, so $\tilde{\gamma}(t_1)$ must lie over the connected component of the point $\tilde{\alpha}^{\max(k_0,k_1)}(q_{{}_H}(\mathscr{e}))$ in $\tilde{\alpha}^{k_0}(U)\cap\tilde{\alpha}^{k_1}(U)$. In particular, $\psi|_{\tilde{\alpha}^{k_0}(q_{{}_H}^{-1}(U))}$ and $\psi|_{\tilde{\alpha}^{k_1}(q_{{}_H}^{-1}(U))}$ must coincide on $\tilde{\gamma}(t_1)$ because $|k_0-k_1|=1$, so we can extend $\beta$ to $[0,t_2]$ by defining $\beta|_{[t_1,t_2]}:=\psi|_{\tilde{\alpha}^{k_1}(q_{{}_H}^{-1}(U))}\circ\tilde{\gamma}|_{[t_1,t_2]}$.

By iterating this procedure, defining \[\beta|_{[t_j,t_{j+1}]}:=\psi|_{\tilde{\alpha}^{k_j}(q_{{}_H}^{-1}(U))}\circ\tilde{\gamma}|_{[t_j,t_{j+1}]}\] for each $j$, we get a well-defined path $\beta$ that must be of the form $\psi\circ\tilde{\gamma}$ if the extension $\psi$ is well-defined. In particular, $\beta$ is a path from $\beta(0)=\varphi^{i_1}((\psi|_{q_{{}_H}^{-1}(U)})(\mathscr{g}_1))$ to $\beta(1)=\varphi^{i_2}((\psi|_{q_{{}_H}^{-1}(U)})(\mathscr{g}_2))$ with $\beta_G=\tilde{\gamma}_G$, so it must be a thinly null-homotopic loop based at \[\psi(\tilde{\alpha}^{i_1}(\mathscr{g}_1))=\beta(0)=\beta(1)=\psi(\tilde{\alpha}^{i_2}(\mathscr{g}_2)).\mbox{\qedhere}\]\end{proof}

\subsection{A remark on the intricacy of the sprawl definition}
To be blunt, the definition of the sprawl is quite involved. We probably should not be too surprised by this: the construction basically encodes all of the information we can get from the local behavior of an automorphism near a distinguished element, so it needs to be at least complicated enough to account for all instances of automorphisms admitting the same local behavior. Still, it can be tempting to imagine that the definition of sprawls can somehow be radically simplified to something that is always easy to implement.

Lamentably, if such a simplification exists, then it has evaded us even after considerable effort spent trying to find it. The following modification of the example from the beginning of the section is useful for understanding why the sprawl definition must be so involved, even for ostensibly straightforward situations.

Let $(\mathscr{G},\omega)$ be the Cartan geometry of type $(\Iso(2),\Orth(2))$ corresponding to the standard Riemannian structure on the unit 2-sphere $M=\mathbb{S}^2$. We choose $\alpha\in\Aut(\mathscr{G},\omega)\simeq\Orth(3)$ to be a nontrivial rotation, and $\mathscr{e}\in\mathscr{G}$ to be an element lying over one of the two fixed points of $\alpha$; we can think of $q_{{}_{\Orth(2)}}(\mathscr{e})$ as the ``north pole'' and the other fixed point $x$ as the ``south pole''. Define $U=D\cup D'\cup V$ as in Figure \ref{advrotexample}, where $D$ is a small $\alpha$-invariant disk around $q_{{}_{\Orth(2)}}(\mathscr{e})$, $D'$ is a small $\alpha$-invariant disk around $x$, and $V$ is the interior of a spherical lune between the two fixed points such that $\alpha(V)\cap V=\emptyset$.

\begin{figure}
\centering\includegraphics[width=0.4\textwidth]{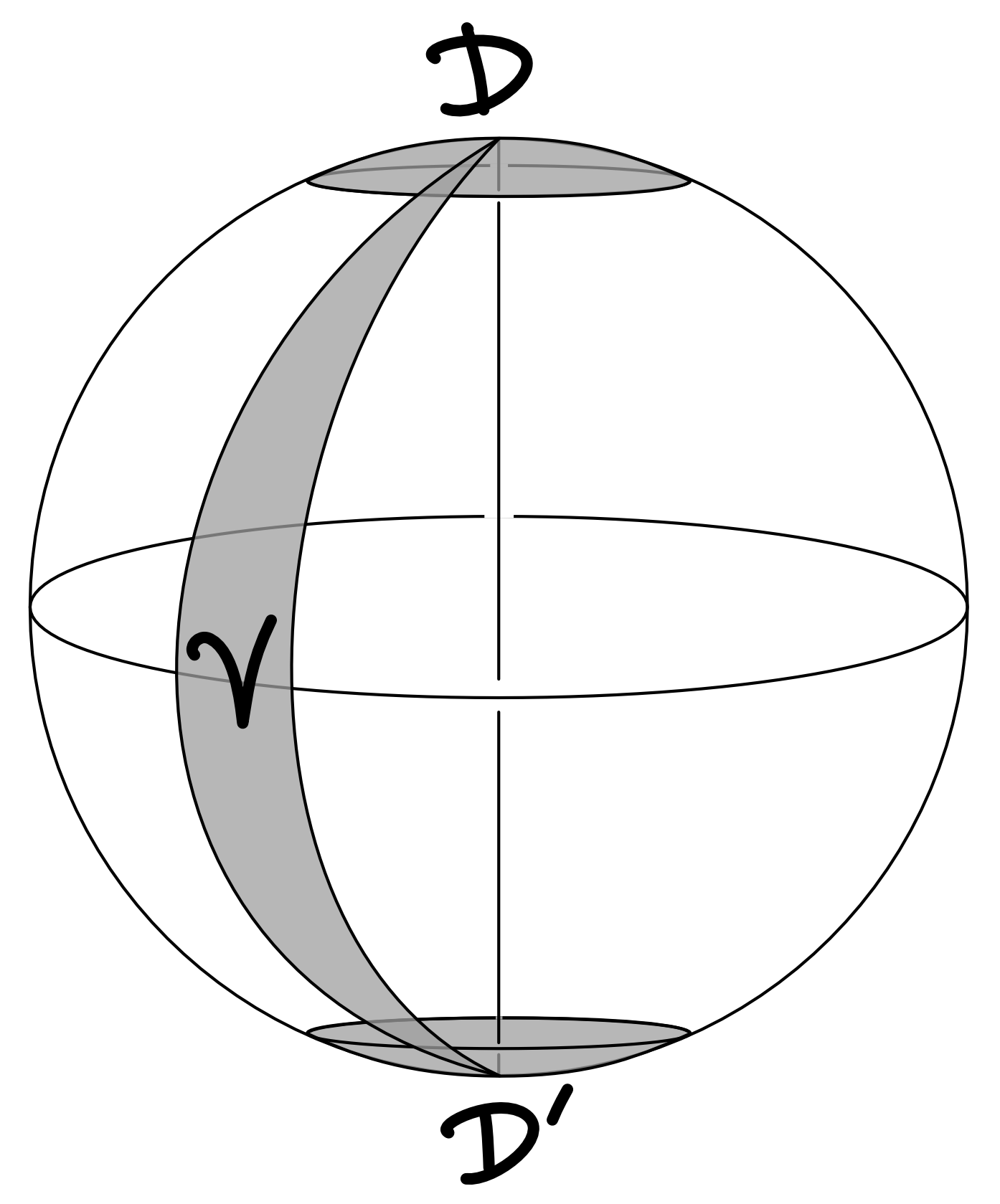}
\caption{The region $U=D\cup D'\cup V$ given by the union of two open disks $D$ and $D'$, one around each fixed point of a rotation on the 2-sphere, with the interior $V$ of a spherical lune connecting them}
\label{advrotexample}
\end{figure}

To compute the sprawl $(\mathscr{F},\sigma^*\omega)$ of $(q_{{}_{\Orth(2)}}^{-1}(U),\omega)$ generated by $\alpha$ from $\mathscr{e}$, note that, whenever $\alpha^{i_1}(V)\cap\alpha^{i_2}(V)\neq\emptyset$, each $y\in\alpha^{i_1}(U)\cap\alpha^{i_2}(U)$ has a path $\delta:[0,1]\to\alpha^{i_1}(U)\cap\alpha^{i_2}(U)$ starting at $\delta(0)=y$ and ending at $\delta(1)=q_{{}_{\Orth(2)}}(\mathscr{e})$; since $q_{{}_{\Orth(2)}}(\mathscr{e})$ happens to be a fixed point for $\alpha$, the concatenation $\delta\star\bar{\delta}$ is therefore a thinly null-homotopic loop that can be incremented from $i_1$ to $i_2$, so $\tilde{\alpha}^{i_1}(\alpha^{-i_1}(y))$ is sprawl-equivalent to $\tilde{\alpha}^{i_2}(\alpha^{-i_2}(y))$.

Thus, if $\alpha^k(V)=V$ for some $k>1$ and $\alpha^i(V)\cap V=\emptyset$ for each $0<i<k$, then the base manifold $q_{{}_{\Orth(2)}}(\mathscr{F})$ will have exactly $k$ disjoint copies of $D'$, given by $\tilde{\alpha}^0(D'),\dots,\tilde{\alpha}^{k-1}(D')$. In this case, $q_{{}_{\Orth(2)}}(\mathscr{F})$ is topologically reminiscent of a starfish with disks attached to the ends of each of its appendages.

If $\alpha^k(V)\cap V\neq\emptyset$ for some $k>1$, but $\alpha^k(V)\neq V$, then things get more complicated. Indeed, the above tells us that, for each $n$, $\tilde{\alpha}^{nk}(D')$ will coincide with $\tilde{\alpha}^0(D')$, but every $\alpha^i(V)$ will eventually have nontrivial intersection with $\alpha^{nk}(V)$ for some $n$, so we must also have $\tilde{\alpha}^i(D')=\tilde{\alpha}^{nk}(D')=\tilde{\alpha}^0(D')$. In other words, $\tilde{\alpha}^0(D')$ is also $\tilde{\alpha}$-invariant, and it is not terribly difficult to see that $(\mathscr{F},\sigma^*\omega)$ must be isomorphic to $(\mathscr{G},\omega)$ in this case. Note, in particular, that $\tilde{\alpha}(x)\sim x$ despite the fact that $x$ is not in the same connected component as $q_{{}_{\Orth(2)}}(\mathscr{e})$ in the intersection $\alpha(U)\cap U$, and we do not inherently even know how many intermediate steps it might take to get the relevant incrementation from $0$ to $1$.


\section{Applications for parabolic geometries}\label{applications}
Because sprawls encapsulate the local behavior of automorphisms via the ``universal property'' described in Theorem \ref{universalproperty}, the construction is remarkably well-suited to describing---and placing strong restrictions on---Cartan geometries with automorphisms admitting particular local behaviors.

One especially fruitful strategy for using the sprawl, which we shall demonstrate in the examples below, is to relate the local behavior and structure to that of the model geometry. For example, if we happen to know that the curvature vanishes over some connected open subset containing a point and its image under a given automorphism, then by Theorem \ref{universalproperty}, the sprawl generated by that automorphism will, if we make the flat neighborhood sufficiently small, coincide with a sprawl generated by an automorphism of the model geometry. Alternatively, since the topology of the sprawl $\mathscr{F}=\mathscr{F}(q_{{}_H}^{-1}(U),\alpha,\mathscr{e})$ does not depend on the underlying choice of Cartan connection $\omega$ on $q_{{}_H}^{-1}(U)\subseteq\mathscr{G}$, if a given \textit{local} automorphism of a Cartan geometry coincides with an automorphism of the model topologically, then we might be able to construct the sprawl from the Klein geometry and then modify the connection on that sprawl to get a non-flat Cartan geometry with a \textit{global} automorphism extending the local one with which we started. Since it is generally not too difficult to compute the sprawl generated by an automorphism inside the Klein geometry, this leads us to several easy results describing Cartan geometries admitting automorphisms with particular local behavior.

For our present purposes, we shall specifically focus on parabolic\linebreak geometries admitting automorphisms with higher-order fixed points. A model geometry $(G,P)$ is called \emph{parabolic} when $G$ is a semisimple Lie group and $P$ is a parabolic subgroup, and a \emph{parabolic geometry} is a Cartan geometry $(\mathscr{G},\omega)$ of parabolic type $(G,P)$ satisfying certain mild curvature restrictions called \emph{regularity} and \emph{normality}; see 3.1 of \cite{CapSlovakPG1} for details. Denoting by $P_+<P$ the nilradical of the parabolic subgroup $P$, an automorphism $\alpha$ of a parabolic geometry $(\mathscr{G},\omega)$ of type $(G,P)$ is said to admit a \emph{higher-order fixed point}\footnote{This terminology comes from situations where $P_+$ is abelian, as is the case for the model geometries corresponding to projective and conformal structures. In such cases, the condition amounts to saying that the automorphism $\alpha$ fixes a point $q_{{}_P}(\mathscr{e})$ of the base manifold $M$ such that that the derivative $\alpha_{*q_{{}_P}(\mathscr{e})}$ of $\alpha$ at $q_{{}_P}(\mathscr{e})$ is equal to the identity transformation on the tangent space $T_{q_{{}_P}(\mathscr{e})}M$.} when $\alpha(\mathscr{e})\in\mathscr{e}P_+$ for some $\mathscr{e}\in\mathscr{G}$.

In many---but not all---cases, the curvature of a parabolic geometry will vanish in the neighborhood of a higher-order fixed point. A general\linebreak technique for showing this was outlined in \cite{CapMelnick2013}, though the result goes all the way back to Lemma 5.6 of \cite{NaganoOchiai} in the projective case. As such, automorphisms with higher-order fixed points provide a plethora of examples with which to apply our strategy of passing to the model geometry to compute the sprawl. Moreover, since higher-order fixed points often place such strong local restrictions on the curvature of the geometry, examples that are not globally flat where higher-order fixed points exist would be especially interesting. Prior to this, there have been constructions with incomplete flows, such as in Section 6.2 of \cite{Frances2007} and Example 4.3 in \cite{KruglikovThe2018}---as well as a result in low regularity in \cite{CapMelnick2021}, which does not have a corresponding Cartan geometry---but no examples of \textit{global} automorphisms with higher-order fixed points in a non-flat Cartan geometry. We will describe how to construct such global examples using sprawls.

\subsection{Non-flat projective structures with higher-order fixed points}
Consider the parabolic model geometry $(\PGL_{m+1}\mathbb{R},P)$ corresponding to $m$-dimensional real projective geometry, where $P$ is the parabolic subgroup \[P:=\left\{\begin{pmatrix}r & \beta \\ 0 & R\end{pmatrix}:r\in\mathbb{R}^\times, \beta^\top\in\mathbb{R}^m, R\in\GLin_m\mathbb{R}\right\},\] which is the stabilizer for a point of $\mathbb{RP}^m$ under the usual action of $\PGL_{m+1}\mathbb{R}$. Cartan geometries of type $(\PGL_{m+1}\mathbb{R},P)$ are well-known to correspond to projective structures on $m$-manifolds under the mild curvature restrictions alluded to above; see either Section 4.1.5 of \cite{CapSlovakPG1} or Chapter 8 of \cite{Sharpe1997}.

Within $P$, we have its nilradical \[P_+:=\left\{\begin{pmatrix}1 & \beta \\ 0 & \mathds{1}\end{pmatrix}:\beta^\top\in\mathbb{R}^m\right\},\] and we want to look at (nontrivial) automorphisms $\alpha\in\Aut(\mathscr{G},\omega)$ with higher-order fixed points, so that $\alpha(\mathscr{e})\in\mathscr{e}P_+$ for some $\mathscr{e}\in\mathscr{G}$. By Lemma 5.6 of \cite{NaganoOchiai}, whenever a parabolic geometry $(\mathscr{G},\omega)$ of type $(\PGL_{m+1}\mathbb{R},P)$ admits such an automorphism $\alpha$, the curvature must vanish over a neighborhood of the higher-order fixed point. This tells us that, for some open neighborhood $U$ of $q_{{}_P}(e)$ in $\mathbb{RP}^m$, which we can choose to make properly convex by restricting if necessary, there is a geometric embedding \[\psi:(q_{{}_P}^{-1}(U),\MC{\PGL_{m+1}\mathbb{R}})\hookrightarrow(\mathscr{G},\omega)\] with $\psi(e)=\mathscr{e}$. Denoting by $a\in P_+$ the isotropy element of $\alpha$ at $\mathscr{e}$, so that $\alpha(\mathscr{e})=\mathscr{e}a$, Theorem \ref{universalproperty} then tells us that $\psi$ extends to a geometric map from $(\mathscr{F},\sigma^*\MC{\PGL_{m+1}\mathbb{R}})$, the sprawl of $(q_{{}_P}^{-1}(U),\MC{\PGL_{m+1}\mathbb{R}})$ generated by $a$ from $e$, to $(\mathscr{G},\omega)$ so that $\psi\circ\tilde{a}=\alpha\circ\psi$.

In this case, we can decompose $\mathscr{F}$ into a particularly useful form.

\begin{lemma}\label{naiveprojsprawl}Suppose that $a\in P_+$ is nontrivial and $U$ is a properly convex open neighborhood of $q_{{}_P}(e)$ in $\mathbb{RP}^m\cong\PGL_{m+1}\mathbb{R}/P$. The sprawl $(\mathscr{F},\sigma^*\MC{\PGL_{m+1}\mathbb{R}})$ of $(q_{{}_P}^{-1}(U),\MC{\PGL_{m+1}\mathbb{R}})$ generated by $a$ from the identity element $e\in\PGL_{m+1}\mathbb{R}$, then, is given by the (disjoint) union \[\mathscr{F}=\mathscr{F}_-\cup q_{{}_P}^{-1}\!\left(U\cap\mathrm{Fix}_{\mathbb{RP}^m}(a)\right)\cup\mathscr{F}_+\] of two disjoint $P$-invariant open subsets \[\mathscr{F}_+:=\left\{\tilde{a}^i(g)\in\mathscr{F}:g\not\in P\text{ and }\lim_{k\rightarrow+\infty}\tilde{a}^{k+i}(q_{{}_P}(g))=\tilde{a}^0(q_{{}_P}(e))\right\}\] and \[\mathscr{F}_-:=\left\{\tilde{a}^i(g)\in\mathscr{F}:g\not\in P\text{ and }\lim_{k\rightarrow-\infty}\tilde{a}^{k+i}(q_{{}_P}(g))=\tilde{a}^0(q_{{}_P}(e))\right\}\] with a copy of the preimage in $\PGL_{m+1}\mathbb{R}$ of the subset of $U$ fixed by $a$.\end{lemma}
\begin{proof}Because the subset $U\cap\mathrm{Fix}_{\mathbb{RP}^m}(a)$ is connected by convexity, it remains unchanged inside the sprawl as a subspace of fixed points for the automorphism $\tilde{a}$. Meanwhile, we can define two $P$-invariant open subsets $\sigma(\mathscr{F}_+)$ and $\sigma(\mathscr{F}_-)$ in $\PGL_{m+1}\mathbb{R}$ by \[\sigma(\mathscr{F}_+):=\left\{g\in\PGL_{m+1}\mathbb{R}:g\not\in P\text{ and }\lim_{k\rightarrow+\infty}a^k(q_{{}_P}(g))=q_{{}_P}(e)\right\}\] and \[\sigma(\mathscr{F}_-):=\left\{g\in\PGL_{m+1}\mathbb{R}:g\not\in P\text{ and }\lim_{k\rightarrow-\infty}a^k(q_{{}_P}(g))=q_{{}_P}(e)\right\}.\] These happen to be the same subset of $\PGL_{m+1}\mathbb{R}$---both are equal to the complement of $q_{{}_P}^{-1}(\mathrm{Fix}_{\mathbb{RP}^m}(a))$ inside of $\PGL_{m+1}\mathbb{R}$---though we will treat them separately because they correspond to disjoint open subsets inside the sprawl. Note that, essentially by definition, $a^k(g)\in q_{{}_P}^{-1}(U)$ for sufficiently large $k>0$ if $g\in\sigma(\mathscr{F}_+)$, and similarly for $\sigma(\mathscr{F}_-)$, so both are genuinely in the image of the sprawl map $\sigma:\mathscr{F}\to\PGL_{m+1}\mathbb{R}$.

Since $a$ is a projective transformation, it preserves lines, and since $a$ fixes the origin $q_{{}_P}(e)\in\mathbb{RP}^m$ to first-order, each line through the origin is invariant under $a$. In particular, for each $g\in\sigma(\mathscr{F}_+)$, there is a geodesic segment $\gamma_g:[0,1]\to\mathbb{RP}^m$ contained in $q_{{}_P}(\{e\}\cup\sigma(\mathscr{F}_+))$ connecting $\gamma_g(1)=q_{{}_P}(g)$ to $\gamma_g(0)=q_{{}_P}(e)$ that gets shrunk into $q_{{}_P}(e)$ under forward iterates of $a$. Therefore, for positive $i_1,i_2\in\mathbb{Z}$ such that $a^{i_1}(g),a^{i_2}(g)\in q_{{}_P}^{-1}(U)$, $\gamma_g$ is contained in both $a^{-i_1}(U)$ and $a^{-i_2}(U)$ by convexity and $\tilde{a}^k(q_{{}_P}(e))\sim\tilde{a}^0(q_{{}_P}(e))$ for all $k\in\mathbb{Z}$, so the concatenation $\bar{\gamma}_g\star\gamma_g$ is a thinly null-homotopic loop incremented from $-i_1$ to $-i_2$ between $a^{-i_1}(a^{i_1}(q_{{}_P}(g)))$ and $a^{-i_2}(a^{i_2}(q_{{}_P}(g)))$, hence $\tilde{a}^{-i_1}(a^{i_1}(g))$ is sprawl-equivalent to $\tilde{a}^{-i_2}(a^{i_2}(g))$. In other words, we get a copy $\mathscr{F}_+$ of $\sigma(\mathscr{F}_+)$ inside $\mathscr{F}$, with $g\in\sigma(\mathscr{F}_+)$ corresponding to the sprawl-equivalence class of elements of the form $\tilde{a}^{-k}(a^k(g))$ with $a^k(g)\in q_{{}_P}^{-1}(U)$ for all $k>0$, and by similar considerations, we get a similar copy $\mathscr{F}_-$ of $\sigma(\mathscr{F}_-)$ in $\mathscr{F}$.

The images of $\mathscr{F}_+$ and $\mathscr{F}_-$ in the double cover $(\SLin_{m+1}^\pm\mathbb{R},\MC{\SLin_{m+1}^\pm\mathbb{R}})$ of the Klein geometry are disjoint, so they must be disjoint in the sprawl. Hence, the result is proven.\mbox{\qedhere}\end{proof}

This decomposition of the sprawl is very useful because the geometric map $\psi$ must be injective when restricted to $\mathscr{F}_+$ and $\mathscr{F}_-$.

\begin{lemma}For the geometric map $\psi:(\mathscr{F},\sigma^*\MC{\PGL_{m+1}\mathbb{R}})\rightarrow(\mathscr{G},\omega)$ such that $\psi\circ\tilde{a}=\alpha\circ\psi$ defined above, $\psi|_{\mathscr{F}_+}$ and $\psi|_{\mathscr{F}_-}$ are injective.\end{lemma}
\begin{proof}Suppose $\psi(\tilde{a}^{i_1}(g_1))=\psi(\tilde{a}^{i_2}(g_2))$ for some $\tilde{a}^{i_1}(g_1),\tilde{a}^{i_2}(g_2)\in\mathscr{F}_+$. Then, by definition, $\tilde{a}^{i_1+k}(g_1)$ and $\tilde{a}^{i_2+k}(g_2)$ must be in $\tilde{a}^0(q_{{}_P}^{-1}(U))$ for all sufficiently large $k>0$, and $\psi$ is a geometric embedding on $\tilde{a}^0(q_{{}_P}^{-1}(U))\cong q_{{}_P}^{-1}(U)$, so $\tilde{a}^{i_1+k}(g_1)\sim\tilde{a}^{i_2+k}(g_2)$, hence $\tilde{a}^{i_1}(g_1)\sim\tilde{a}^{i_2}(g_2)$. Similarly, if $\psi(\tilde{a}^{i_1}(g_1))=\psi(\tilde{a}^{i_2}(g_2))$ for some $\tilde{a}^{i_1}(g_1),\tilde{a}^{i_2}(g_2)\in\mathscr{F}_-$, then $\tilde{a}^{i_1-k}(g_1)\sim\tilde{a}^{i_2-k}(g_2)$ for all sufficiently large $k>0$, so we again get $\tilde{a}^{i_1}(g_1)\sim\tilde{a}^{i_2}(g_2)$.\mbox{\qedhere}\end{proof}

Thus, by using Lemma 5.6 of \cite{NaganoOchiai} together with our Theorem \ref{universalproperty}, we get copies of large open subsets of the Klein geometry embedded into every Cartan geometry of type $(\PGL_{m+1}\mathbb{R},P)$ (satisfying certain mild curvature restrictions) admitting a nontrivial automorphism with a higher-order fixed point.

While this does determine a large chunk of each parabolic geometry of type $(\PGL_{m+1}\mathbb{R},P)$ with a nontrivial higher-order fixed point, we cannot guarantee that the geometry is globally flat without making further assumptions. In particular, we can construct non-flat examples as follows: take two copies of the sprawl $\mathscr{F}$ and $\mathscr{F}'$, glue them together by $\tilde{a}$-equivariantly identifying $\mathscr{F}'_+$ with $\mathscr{F}_-$ (in the same way that $\mathscr{F}_+$ is identified with $\mathscr{F}_-$ when mapped into the model geometry over $\mathbb{RP}^m$), remove the higher-order fixed point $q_{{}_P}(\mathscr{e}')$ from $\mathscr{F}'$, and then consider smooth $\tilde{a}$-invariant deformations of the Cartan connection that only vary over $\mathscr{F}'_-$. Since we have removed the higher-order fixed point $q_{{}_P}(\mathscr{e}')$ from the second copy $\mathscr{F}'$ of the sprawl, the curvature of the connection over $\mathscr{F}'_-$ is no longer forced to be globally flat, and we can get several such examples by considering projective structures on $\mathbb{R}^m\cong q_{{}_P}(\mathscr{F}'_-)$ that are invariant under a translation on $\mathbb{R}^m$, which we can identify with $\tilde{a}$ on $q_{{}_P}(\mathscr{F}'_-)$, and whose curvature vanishes outside of a proper invariant open subset. We have illustrated a schematic diagram for what this looks like in Figure \ref{projpatchpic}.

\begin{figure}
\centering\includegraphics[width=0.8\textwidth]{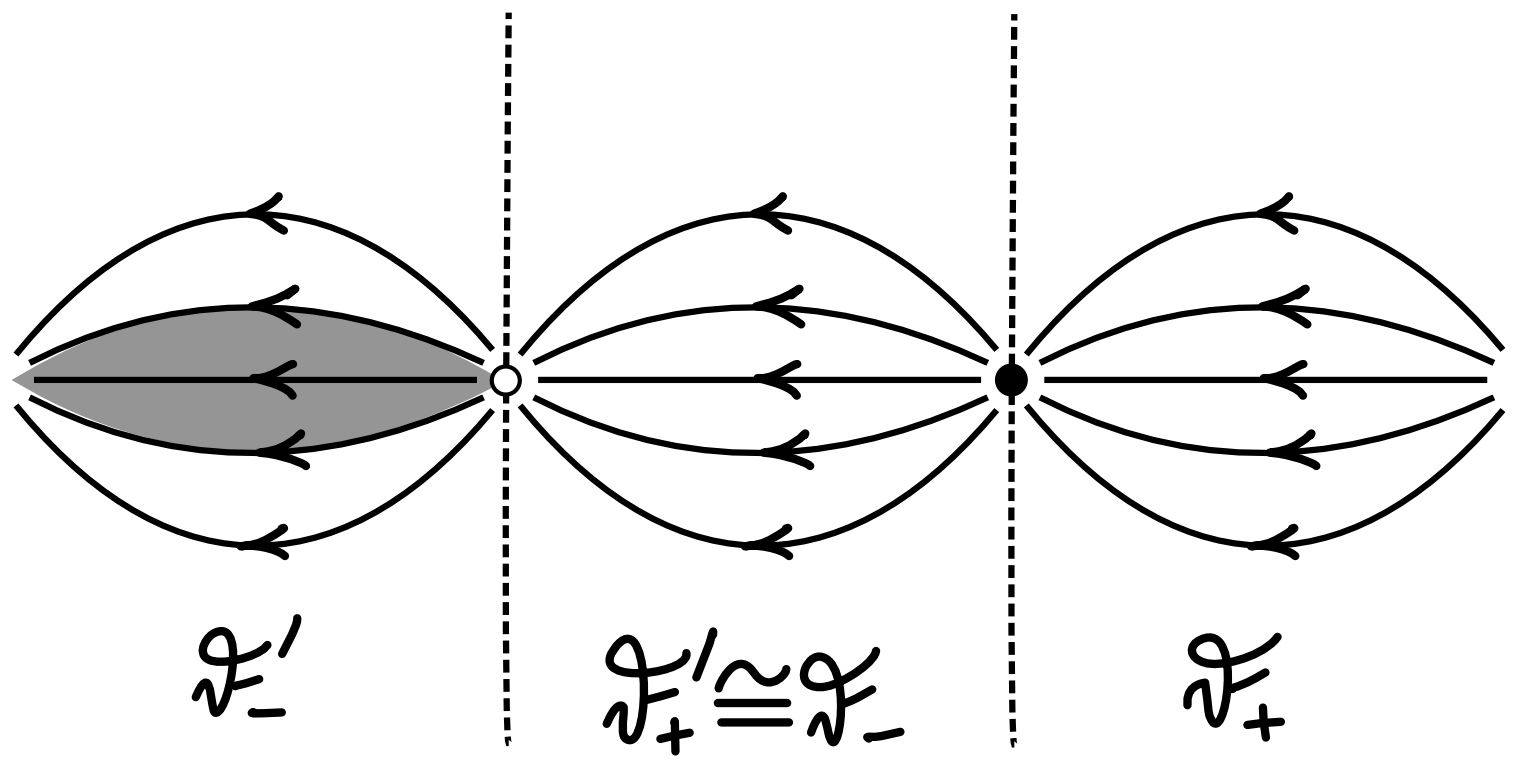}
\caption{A schematic diagram for a projective geometry admitting an automorphism with a higher-order fixed point, with a non-flat region highlighted in gray}
\label{projpatchpic}
\end{figure}

\subsection{Non-flat conformal Lorentzian structures with higher-order fixed points}
In this case, we consider the parabolic model geometry $(\PO(\mathrm{h}_{2,m}),P)$, where $\mathrm{h}_{2,m}$ is the symmetric bilinear form on $\mathbb{R}^{m+2}$ of signature $(2,m)$ with quadratic form given by \[\begin{bmatrix}x_0 \\ x \\ x_{m+1}\end{bmatrix}\mapsto 2x_0x_{m+1}+x^\top I_{1,m-1} x,\] with $I_{1,m-1}$ the diagonal $m\times m$ matrix with $+1$ on the first diagonal entry and $-1$ on all of the others, and $P$ is the parabolic subgroup \[P:=\left\{\begin{pmatrix}r & r\beta & -\tfrac{r}{2}\beta I_{1,m-1}\beta^\top \\ 0 & A & -AI_{1,m-1}\beta^\top \\ 0 & 0 & r^{-1}\end{pmatrix}\!:\,\begin{array}{l} r\in\mathbb{R}^\times,\, \beta^\top\in\mathbb{R}^{m}, \\ \text{and }A\in\Orth(1,m-1)\end{array}\right\}\] with nilradical \[P_+:=\left\{\left(\begin{smallmatrix}1 & \beta & -\tfrac{1}{2}\beta I_{1,m-1}\beta^\top \\ 0 & \mathds{1} & -I_{1,m-1}\beta^\top \\ 0 & 0 & 1\end{smallmatrix}\right):\beta^\top\in\mathbb{R}^{m}\right\}.\] Here, the homogeneous space $\PO(\mathrm{h}_{2,m})/P$ is $\PO(\mathrm{h}_{2,m})$-equivariantly diffeomorphic to the projectivized null-cone \[\mathrm{Null}(\mathrm{h}_{2,m}):=\left\{\mathbb{R}^\times u\in\mathbb{RP}^{m+1}:\mathrm{h}_{2,m}(u,u)=0\right\}\] of $\mathrm{h}_{2,m}$ in $\mathbb{RP}^{m+1}$, and for $m\geq 3$, Cartan geometries of type $(\PO(\mathrm{h}_{2,m}),P)$ with the mild curvature restrictions alluded to above are well-known to correspond to conformal Lorentzian structures on $m$-dimensional manifolds; see Section 4.1.2 of \cite{CapSlovakPG1}.

We would like to construct an example of a parabolic geometry of type $(\PO(\mathrm{h}_{2,m}),P)$ that is not globally flat but still admits a nontrivial automorphism with a higher-order fixed point. Fortunately, we already have two local examples for this, given in Section 6.2 of \cite{Frances2007}. Both of these local examples are constructed in a manner similar to that of the non-flat projective example above, but instead of needing to glue sprawls together, Frances manages to directly construct deformations of the conformal Lorentzian structure over Minkowski space that are locally invariant under the flows of a vector field with a higher-order fixed point at the origin.

Writing $G_-$ for the abelian horospherical subgroup \[G_-:=\left\{\begin{pmatrix}1 & 0 & 0 \\ x & \mathds{1} & 0 \\ -\tfrac{x^\top I_{1,m-1}x}{2} & -x^\top I_{1,m-1} & 1\end{pmatrix}:x\in\mathbb{R}^{m}\right\}\simeq\mathbb{R}^m,\] the flat conformal Lorentzian structure on Minkowski space is given by the restriction of the Maurer--Cartan form of $\PO(\mathrm{h}_{2,m})$ to the open subset $G_-P\subseteq G$ over $q_{{}_P}(G_-)\cong G_-\cong\mathbb{R}^m$. This gives an embedding of Minkowski space into the Klein geometry for $(\PO(\mathrm{h}_{2,m}),P)$, and the flows of the vector fields in Frances's local examples correspond to the restrictions $t\mapsto\Lt{\exp(tY)}|_{G_-P}$ of left-translations by one-parameter subgroups in $P_+$ to the open subset $G_-P$. Unfortunately, these flows are incomplete over Minkowski space, so they only determine local automorphisms, and the invariant deformations over Minkowski space do not extend to the full Klein geometry, where the flows would become complete.

To get the desired examples with \textit{global} automorphisms admitting higher-order fixed points, we will use the following lemma.

\begin{lemma}\label{locext} Let $(\mathscr{G},\omega)$ be a Cartan geometry of type $(G,H)$ over $M$ with automorphism $\alpha$. Suppose that there is an $H$-invariant open subset $\mathscr{G}'\subseteq\mathscr{G}$ and an open subset $U\subseteq q_{{}_H}(\mathscr{G}')$ such that $\alpha(\mathscr{e})\in q_{{}_H}^{-1}(U)$ for some $\mathscr{e}\in q_{{}_H}^{-1}(U)$ and $\alpha(U)\subseteq q_{{}_H}(\mathscr{G}')$. Furthermore, suppose that sprawl-equivalence for the sprawl $(\mathscr{F},\sigma^*\omega)$ of $(q_{{}_H}^{-1}(U),\omega)$ generated by $\alpha$ from $\mathscr{e}$ coincides with the na\"{i}ve gluing from Section \ref{sprawl1}. Then, if $\omega'$ is a Cartan connection of type $(G,H)$ on $\mathscr{G}'$ for which $\alpha$ is a local automorphism, then there exists a Cartan connection $\tilde{\omega}'$ on $\mathscr{F}$ such that $\tilde{\alpha}\in\Aut(\mathscr{F},\tilde{\omega}')$ and $\tilde{\alpha}^i:(q_{{}_H}^{-1}(U),\omega)\hookrightarrow(\mathscr{F},\tilde{\omega}')$ is a geometric embedding for each $i\in\mathbb{Z}$.\end{lemma}
\begin{proof}Define $\tilde{\omega}'$ on $\mathscr{F}$ by $\tilde{\omega}'_{\tilde{\alpha}^i(\mathscr{g})}:=\omega'_\mathscr{g}\circ(\tilde{\alpha}^i)^{-1}_*$ for each $\tilde{\alpha}^i(\mathscr{g})\in\mathscr{F}$. Since $\omega'$ is a Cartan connection on $\mathscr{G}'$ and each of the relabeling maps $\tilde{\alpha}^i:q_{{}_H}^{-1}(U)\to\tilde{\alpha}^i(q_{{}_H}^{-1}(U))$ is an $H$-equivariant diffeomorphism, $\tilde{\omega}'$ must also be a Cartan connection if it is well-defined, and it must also be $\tilde{\alpha}$-invariant because \[(\tilde{\alpha}^*\tilde{\omega}')_{\tilde{\alpha}^i(\mathscr{g})}=\tilde{\omega}'_{\tilde{\alpha}^{i-1}(\mathscr{g})}\circ\tilde{\alpha}_*=\omega'_\mathscr{g}\circ(\tilde{\alpha}^{i-1})^{-1}_*\circ\tilde{\alpha}_*=\omega'_\mathscr{g}\circ(\tilde{\alpha}^i)^{-1}_*=\tilde{\omega}'_{\tilde{\alpha}^i(\mathscr{g})}.\] Thus, all we need to do is prove that $\tilde{\omega}'$ is well-defined.

Because we are assuming that sprawl-equivalence coincides with the equivalence relation on $\bigsqcup_{i\in\mathbb{Z}}\tilde{\alpha}^i(q_{{}_H}^{-1}(U))$ induced by the na\"{i}ve gluing in this case, each identification between the different copies $\tilde{\alpha}^k(q_{{}_H}^{-1}(U))$ of $q_{{}_H}^{-1}(U)$ comes from iterating identifications between adjacent copies $\tilde{\alpha}^i(q_{{}_H}^{-1}(U))$ and $\tilde{\alpha}^{i+1}(q_{{}_H}^{-1}(U))$ of the form $\tilde{\alpha}^{i}(\mathscr{g})\sim\tilde{\alpha}^{i+1}(\alpha^{-1}(\mathscr{g}))$, for $q_{{}_H}(\mathscr{g})$ in the connected component of $U\cap\alpha(U)$ containing $\alpha(q_{{}_H}(\mathscr{e}))$. In other words, we just need to show that $\tilde{\omega}'_{\tilde{\alpha}^i(\mathscr{g})}=\tilde{\omega}'_{\tilde{\alpha}^{i+1}(\alpha^{-1}(\mathscr{g}))}$ whenever $q_{{}_H}(\mathscr{g})$ lies in the connected component of $U\cap\alpha(U)$ containing $\alpha(q_{{}_H}(\mathscr{e}))$, which follows from the fact that $\alpha^*\omega'=\omega'$ wherever $\alpha$ is well-defined on $\mathscr{G}'$.\mbox{\qedhere}\end{proof}

By Lemma \ref{locext}, we can extend the local examples of \cite{Frances2007} to global examples if we can show that the sprawl generated by left-translation by $a\in P_+$ in the Klein geometry admits a sprawl whose sprawl-equivalence reduces to the na\"{i}ve gluing. By Theorem \ref{universalproperty}, this amounts to showing that the quotient space of the na\"ive gluing is Hausdorff.

Let us outline how to do this in the case of the first example from \cite{Frances2007}, whose flows correspond to left-translation by the one-parameter subgroup \[\varphi:t\mapsto\exp\left(t\left(\begin{smallmatrix}0 & e_1^\top & 0 \\ 0 & 0 & -e_1 \\ 0 & 0 & 0\end{smallmatrix}\right)\right)=\left(\begin{smallmatrix}1 & te_1^\top & -t^2/2 \\ 0 & \mathds{1} & -te_1 \\ 0 & 0 & 1\end{smallmatrix}\right)\] restricted to the open subset corresponding to Minkowski space. Pick a small open neighborhood $U$ of the origin $q_{{}_P}(e)$ such that, for each $x\in U$ contained in the null-cone $C$ through $q_{{}_P}(e)$, $\varphi_t(x)\in U$ for either all $t\geq 0$ or all $t\leq 0$.
Then, since we have \[\varphi_t\left(\begin{smallmatrix}1 \\ x \\ 0\end{smallmatrix}\right)=\left(\begin{smallmatrix}1+tx_1 \\ x \\ 0\end{smallmatrix}\right)\text{ for each }x=\left[\begin{smallmatrix}x_1 \\ \vdots \\ x_m\end{smallmatrix}\right]\text{ such that }x^\top I_{1,m-1}x=0,\] we get a null-cone $\tilde{C}$ through $\widetilde{\varphi_1}^0(q_{{}_P}(e))$ in $\bigsqcup_{i\in\mathbb{Z}}\widetilde{\varphi_1}^i(U)/\sim_\text{naive}$, with two $\widetilde{\varphi_1}$-invariant halves \[\tilde{C}_+=\{\widetilde{\varphi_1}^i(x):x\in C\cap U\text{ and }\varphi_t(x)\in U\text{ for all }t\geq 0\}\] and \[\tilde{C}_-=\{\widetilde{\varphi_1}^i(x):x\in C\cap U\text{ and }\varphi_t(x)\in U\text{ for all }t\leq 0\}.\] Both $\tilde{C}_+\setminus\{\widetilde{\varphi_1}^0(q_{{}_P}(e))\}$ and $\tilde{C}_-\setminus\{\widetilde{\varphi_1}^0(q_{{}_P}(e))\}$ are naturally isomorphic to $C\setminus\{q_{{}_P}(e)\}$ in $\PO(\mathrm{h}_{2,m})/P$, by $\widetilde{\varphi_1}^i(x)\mapsto(\varphi_1)^i(x)=\varphi_i(x)$. The complement of $\tilde{C}$ inside $\bigsqcup_{i\in\mathbb{Z}}\widetilde{\varphi_1}^i(U)/\sim_\text{naive}$ is a disjoint union of three $\widetilde{\varphi_1}$-invariant open subsets \[\begin{array}{c}q_{{}_P}(\mathscr{F}_+^{>0}):=\{\widetilde{\varphi_1}^i(x):x^\top I_{1,m-1}x>0\text{ and }x_1>0\}, \\ q_{{}_P}(\mathscr{F}_+^{<0}):=\{\widetilde{\varphi_1}^i(x):x^\top I_{1,m-1}x>0\text{ and }x_1<0\}, \text{ and} \\ q_{{}_P}(\mathscr{F}_-):=\{\widetilde{\varphi_1}^i(x):x^\top I_{1,m-1}x<0\},\end{array}\] on each of which the sprawl-equivalence reduces to the na\"ive gluing, analogous to the situation in Lemma \ref{naiveprojsprawl}.

It remains, then, to show that distinct points $\widetilde{\varphi_1}^{i_1}(x_1)$ and $\widetilde{\varphi_1}^{i_2}(x_2)$ on $\tilde{C}$ are separable by neighborhoods. If $\widetilde{\varphi_1}^{i_1}(x_1)$ and $\widetilde{\varphi_1}^{i_2}(x_2)$ are both in $\tilde{C}_+$ or $\tilde{C}_-$, then flowing by $\varphi_t$ allows us to pull both points inside of $\widetilde{\varphi_1}^0(U)$, which is an embedded Hausdorff subspace, hence they must be separated by neighborhoods. If $\widetilde{\varphi_1}^{i_1}(x_1)\in\tilde{C}_+$ and $\widetilde{\varphi_1}^{i_2}(x_2)\in\tilde{C}_-$, then every neighborhood of $\widetilde{\varphi_1}^{i_1}(x_1)$ nontrivially intersects $q_{{}_P}(\mathscr{F}_+^{>0})$ and every neighborhood of $\widetilde{\varphi_1}^{i_2}(x_2)$ nontrivially intersects $q_{{}_P}(\mathscr{F}_+^{<0})$, so we must have $\varphi_1^{i_1}(x_1)=\varphi_1^{i_2}(x_2)=q_{{}_P}(e)$, and hence \[\widetilde{\varphi_1}^{i_1}(x_1)\sim_\text{naive}\widetilde{\varphi_1}^{i_2}(x_2)\sim_\text{naive}\widetilde{\varphi_1}^0(q_{{}_P}(e)),\] for them not to be separable by neighborhoods. Therefore, distinct points of $\tilde{C}$ are separable by neighborhoods in $\bigsqcup_{i\in\mathbb{Z}}\widetilde{\varphi_1}^i(U)/\sim_\text{naive}$, so it follows that $\bigsqcup_{i\in\mathbb{Z}}\widetilde{\varphi_1}^i(U)/\sim_\text{naive}$ is Hausdorff.

\subsection{Non-flat path geometries with higher-order fixed points}
Finally, consider the parabolic model geometry $(\PGL_{m+1}\mathbb{R},Q)$, where $Q$ is the parabolic subgroup \[Q:=\left\{\begin{pmatrix}r & \tau & q \\ 0 & s & p \\ 0 & 0 & A\end{pmatrix}:\begin{array}{c} r,s\in\mathbb{R}^\times,\tau\in\mathbb{R}, \\ p^\top,q^\top\in\mathbb{R}^{m-1}, \\ A\in\GLin_{m-1}\mathbb{R}\end{array}\right\}\] with nilradical \[Q_+:=\left\{\begin{pmatrix}1 & \tau & q \\ 0 & 1 & p \\ 0 & 0 & \mathds{1}\end{pmatrix}:\tau\in\mathbb{R}\text{ and } p^\top,q^\top\in\mathbb{R}^{m-1}\right\}.\] Note that $Q$ is a subgroup of the parabolic subgroup $P<\PGL_{m+1}\mathbb{R}$\linebreak defined above for projective geometries. The underlying homogeneous space $\PGL_{m+1}\mathbb{R}/Q$ corresponds to the projectivized tangent bundle $\mathbb{P}(T\mathbb{RP}^m):=\{\mathbb{R}^\times u: u\in T\mathbb{RP}^m\}$, and parabolic geometries of type $(\PGL_{m+1}\mathbb{R},Q)$ correspond to geometric structures called \emph{(generalized) path geometries}; see Section 4.4.3 of \cite{CapSlovakPG1}.

Similar to the conformal Lorentzian case above, we would like to construct a parabolic geometry of type $(\PGL_{m+1}\mathbb{R},Q)$ with a global automorphism admitting a higher-order fixed point, and we have a local example given by Example 4.3 of \cite{KruglikovThe2018}. In the language of \cite{Erickson2022-2}, if we denote by \[G_-:=\left\{\begin{pmatrix}1 & 0 & 0 \\ t & 1 & 0 \\ x & v & \mathds{1}\end{pmatrix}:t\in\mathbb{R}\text{ and } x,v\in\mathbb{R}^{m-1}\right\}\] a horospherical subgroup such that $G_-Q$ is open in $\PGL_{m+1}\mathbb{R}$ and by $\kgf$ the Killing form on $\mathfrak{pgl}_{m+1}\mathbb{R}$, the example from \cite{KruglikovThe2018} is a ``curvature tree'' grown from an element of the form $\ell\Omega\in\Lambda^2(\mathfrak{pgl}_{m+1}\mathbb{R}/\mathfrak{q})^\vee\otimes\mathfrak{pgl}_{m+1}\mathbb{R}$, for some $\ell\neq 0$ and \[\Omega=\tfrac{1}{4(m+1)^2}\left(\begin{smallmatrix}0 & 0 & 0 \\ 0 & 0 & e_1^\top \\ 0 & 0 & 0\end{smallmatrix}\right)_\kgf\wedge\left(\begin{smallmatrix}0 & 0 & e_1^\top \\ 0 & 0 & 0 \\ 0 & 0 & 0\end{smallmatrix}\right)_\kgf\otimes\left(\begin{smallmatrix}0 & 0 & 0 \\ 0 & 0 & 0 \\ 0 & 0 & e_{m-1}e_1^\top\end{smallmatrix}\right),\] whose Cartan connection $\omega_{\ell\Omega}$ we can view as a deformation of the\linebreak restriction of the Maurer--Cartan form in the Klein geometry to $G_-Q$. Denoting by $\Ad_q\cdot\Omega$ the natural action of $q\in Q$ on the element $\Omega$ in $\Lambda^2(\mathfrak{pgl}_{m+1}\mathbb{R}/\mathfrak{q})^\vee\otimes\mathfrak{pgl}_{m+1}\mathbb{R}$ induced by the adjoint representation, we have \begin{align*}a_\tau\cdot\Omega & =\tfrac{1}{4(m+1)^2}\Ad_{a_\tau}\left(\begin{smallmatrix}0 & 0 & 0 \\ 0 & 0 & e_1^\top \\ 0 & 0 & 0\end{smallmatrix}\right)_\kgf\wedge\Ad_{a_\tau}\left(\begin{smallmatrix}0 & 0 & e_1^\top \\ 0 & 0 & 0 \\ 0 & 0 & 0\end{smallmatrix}\right)_\kgf\otimes\Ad_{a_\tau}\left(\begin{smallmatrix}0 & 0 & 0 \\ 0 & 0 & 0 \\ 0 & 0 & e_{m-1}e_1^\top\end{smallmatrix}\right) \\ & =\tfrac{1}{4(m+1)^2}\left(\begin{smallmatrix}0 & 0 & \tau e_1^\top \\ 0 & 0 & e_1^\top \\ 0 & 0 & 0\end{smallmatrix}\right)_\kgf\wedge\left(\begin{smallmatrix}0 & 0 & e_1^\top \\ 0 & 0 & 0 \\ 0 & 0 & 0\end{smallmatrix}\right)_\kgf\otimes\left(\begin{smallmatrix}0 & 0 & 0 \\ 0 & 0 & 0 \\ 0 & 0 & e_{m-1}e_1^\top\end{smallmatrix}\right) \\ & =\Omega\end{align*} for \[a_\tau:=\left(\begin{smallmatrix}1 & \tau & 0 \\ 0 & 1 & 0 \\ 0 & 0 & \mathds{1}\end{smallmatrix}\right),\] so left-translations by such elements $a_\tau\in Q_+$ give local automorphisms of the curvature tree $(G_-Q,\omega_{\ell\Omega})$. Since $a_\tau(e)=ea_\tau=a_\tau\in G_-Q$, this local automorphism has a higher-order fixed point at $q_{{}_Q}(e)$.

To get an extended, global geometry from this, we again just need to apply Lemma \ref{locext}, by showing that there exists an open neighborhood $U$ of $q_{{}_Q}(e)$ such that sprawl-equivalence for the sprawl of $(q_{{}_Q}^{-1}(U),\MC{\PGL_{m+1}\mathbb{R}})$ generated by $a_\tau$ from $e$ coincides with the na\"ive gluing. For this purpose, note that the inclusion $Q<P$ determines a natural map \[q_P:\PGL_{m+1}\mathbb{R}/Q\to\PGL_{m+1}\mathbb{R}/P,\,q_{{}_Q}(g)\mapsto q_{{}_P}(g),\] corresponding to the natural projection from $\mathbb{P}(T\mathbb{RP}^m)$ to $\mathbb{RP}^m$, and choose our open neighborhood $U\subseteq G_-Q$ so that $q_P(U)$ is a properly convex open neighborhood of $q_P(q_{{}_Q}(e))=q_{{}_P}(e)\in\mathbb{RP}^m$. Then, we have \[q_{{}_Q}^{-1}(U)\subseteq q_{{}_P}^{-1}(q_P(U))\subseteq\PGL_{m+1}\mathbb{R},\] so viewing the sprawl $(\mathscr{F},\sigma^*\MC{\PGL_{m+1}\mathbb{R}})$ of $q_{{}_P}^{-1}(q_P(U))$ generated by $a_\tau$ from $e$ as a Cartan geometry of type $(\PGL_{m+1}\mathbb{R},Q)$ over $\mathscr{F}/Q$, we get a geometric embedding from $(\mathscr{F}',\sigma^*\MC{\PGL_{m+1}\mathbb{R}})$, the sprawl of $q_{{}_Q}^{-1}(U)$ generated by $a_\tau$ from $e$, into $(\mathscr{F},\sigma^*\MC{\PGL_{m+1}\mathbb{R}})$. In particular, since we have seen from Lemma \ref{naiveprojsprawl} that the sprawl equivalence for $\mathscr{F}$ coincides with the na\"ive gluing, it follows that the same is true of $\mathscr{F}'$, so applying Lemma \ref{locext} gives us a global extension of Example 4.3 from \cite{KruglikovThe2018}.

\section*{Appendix: the holonomy group of the sprawl}
Anticipating a growing interest in techniques utilizing the holonomy groups of Cartan geometries, we have decided to include the following supplementary result. Throughout the proof, we will unabashedly use ideas and terms from \cite{HolonomyPaper}, particularly for developments of points and their relations to automorphisms.

Suppose we are given a Cartan geometry $(\mathscr{G},\omega)$ of type $(G,H)$ over a smooth manifold $M$, together with a distinguished element $\mathscr{e}\in\mathscr{G}$, an automorphism $\alpha\in\Aut(\mathscr{G},\omega)$, and a connected open subset $U\subseteq M$ containing $q_{{}_H}(\mathscr{e})$ and $\alpha(q_{{}_H}(\mathscr{e}))$.

\begin{proposition}For $(\mathscr{F},\sigma^*\omega)$ the sprawl of $(q_{{}_H}^{-1}(U),\omega)$ generated by $\alpha\in\Aut(\mathscr{G},\omega)$ from $\mathscr{e}$, let $a\in G$ be a development of $\alpha(\mathscr{e})$ from $\mathscr{e}$ as elements of $(q_{{}_H}^{-1}(U),\omega)$. The holonomy group $\Hol_\mathscr{e}(\mathscr{F},\sigma^*\omega)$ of the sprawl is the smallest subgroup of $G$ containing $\Hol_\mathscr{e}(q_{{}_H}^{-1}(U),\omega)$ that is normalized by $a$.\end{proposition}
\begin{proof}

Suppose $\gamma:[0,1]\to\mathscr{F}$ is a path lying over a loop $q_{{}_H}(\gamma)$ in $q_{{}_H}(\mathscr{F})$, with $\gamma(0)=\mathscr{e}$ and $\gamma(1)=\gamma(0)h_\gamma$. We want to compute $\gamma_G(1)h_\gamma^{-1}$. First, we will show that we can assume $\gamma$ lies over a loop incremented from $0$ to $0$, and then we will compute what $\gamma_G(1)h_\gamma^{-1}$ can be.

Let us break $\gamma$ into a concatenation of segments $\gamma=\gamma_0\star\cdots\star\gamma_{\ell-1}$ such that, for each $0\leq j<\ell$, $\gamma_j([0,1])\subseteq\tilde{\alpha}^{k_j}(q_{{}_H}^{-1}(U))$ for some $k_j\in\mathbb{Z}$. By definition of sprawl-equivalence, for \[\gamma_j(1)=\gamma_{j+1}(0)\in\tilde{\alpha}^{k_j}(q_{{}_H}^{-1}(U))\cap\tilde{\alpha}^{k_{j+1}}(q_{{}_H}^{-1}(U)),\] there must be a thinly null-homotopic loop $\gamma_{j,j+1}$ based at $\gamma_j(1)$ with $q_{{}_H}(\gamma_{j,j+1})$ incremented from $k_j$ to $k_{j+1}$. In particular, the modified path \[\gamma_0\star\gamma_{0,1}\star\gamma_1\star\gamma_{1,2}\star\cdots\star\gamma_{\ell-1}\star\gamma_{\ell-1,\ell}\star\gamma_\ell\] in $\mathscr{F}$ descends to a loop in $q_{{}_H}(\mathscr{F})$ with an incrementation from $0$ to $k_\ell$, and since the holonomy of a thinly null-homotopic loop is always trivial, \[(\gamma_0\star\gamma_{0,1}\star\gamma_1\star\gamma_{1,2}\star\cdots\star\gamma_{\ell-1}\star\gamma_{\ell-1,\ell}\star\gamma_\ell)_G(1)=\gamma_G(1).\] Thus, without loss of generality, we may assume that $\gamma$ lies over a loop $q_{{}_H}(\gamma)$ that is incremented from $0$ to $k_\ell$ for some $k_\ell\in\mathbb{Z}$. Moreover, since \[q_{{}_H}(\mathscr{e})=q_{{}_H}(\gamma(0))=q_{{}_H}(\gamma(1))\in\tilde{\alpha}^0(U)\cap\tilde{\alpha}^{k_\ell}(U),\] there must be a thinly null-homotopic loop $\gamma'$ based $\mathscr{e}$ such that $q_{{}_H}(\gamma')$ is incremented from $0$ to $k_\ell$. Concatenating $\gamma$ with $\Rt{h_\gamma}(\overline{\gamma'})$, we again get a path $\gamma\star\Rt{h_\gamma}(\overline{\gamma'})$ lying over a loop in $q_{{}_H}(\mathscr{F})$, this time with an incrementation from $0$ back to $0$, such that $(\gamma\star\Rt{h_\gamma}(\overline{\gamma'}))_G(1)=\gamma_G(1)$. Without loss of generality, we may therefore assume that $\gamma$ lies over a loop $q_{{}_H}(\gamma)$ with an incrementation from $0$ to $0$.

Let this incrementation of the loop $q_{{}_H}(\gamma)$ from $0$ to $0$ be given by the partition $0=t_0<\cdots<t_\ell=1$ and the finite integer sequence $k_0=0,\dots,k_{\ell-1}=0\in\mathbb{Z}$. By definition, $\gamma(t_{j+1})$ lies over the connected component of $\tilde{\alpha}^{k_j}(U)\cap\tilde{\alpha}^{k_{j+1}}(U)$ containing $q_{{}_H}(\tilde{\alpha}^{\max(k_j,k_{j+1})}(\mathscr{e}))$, so there exist paths \[\beta_{j+1}:[0,1]\to\tilde{\alpha}^{k_j}(q_{{}_H}^{-1}(U))\cap\tilde{\alpha}^{k_{j+1}}(q_{{}_H}^{-1}(U))\] with $\beta_{j+1}(0)=\gamma(t_{j+1})$ and $\beta_{j+1}(1)b_{j+1}=\tilde{\alpha}^{\max(k_j,k_{j+1})}(\mathscr{e})$ for some $b_{j+1}\in H$. In particular, $\beta_{j+1}\star\overline{\beta_{j+1}}$ is a thinly null-homotopic loop in $\tilde{\alpha}^{k_j}(q_{{}_H}^{-1}(U))\cap\tilde{\alpha}^{k_{j+1}}(q_{{}_H}^{-1}(U))$, so we may again construct a modified path \[\gamma|_{[0,t_1]}\star\beta_1\star\overline{\beta_1}\star\cdots\star\gamma|_{[t_{\ell-2},t_{\ell-1}]}\star\beta_{\ell-1}\star\overline{\beta_{\ell-1}}\star\gamma|_{[t_{\ell-1},1]}\] with the same total development as $\gamma$; this tells us that we may further assume, without loss of generality, that $\gamma(t_{j+1})b_{j+1}=\alpha^{\max(k_j,k_{j+1})}(\mathscr{e})$ for some $b_{j+1}\in H$ for each $0\leq j<\ell-1$.

With this, each segment $\gamma|_{[t_j,t_{j+1}]}$ with $0\leq j<\ell-1$ is a path from $\gamma(t_j)=\tilde{\alpha}^{\max(k_{j-1},k_j)}(\mathscr{e})b_j^{-1}$ to $\gamma(t_{j+1})=\tilde{\alpha}^{\max(k_j,k_{j+1})}(\mathscr{e})b_{j+1}^{-1}$, so since the space of possible developments from $\tilde{\alpha}^{\max(k_{j-1},k_j)}(\mathscr{e})$ to $\tilde{\alpha}^{\max(k_j,k_{j+1})}(\mathscr{e})$ is just \[\Hol_\mathscr{e}(q_{{}_H}^{-1}(U),\omega)a^{\max(k_j,k_{j+1})-\max(k_{j-1},k_j)}=\Hol_\mathscr{e}(q_{{}_H}^{-1}(U),\omega)a^{\frac{1}{2}(k_{j+1}-k_{j-1})},\] we must have \[(\gamma|_{[t_j,t_{j+1}]})_G(t_{j+1})=b_j\eta_j a^{\frac{1}{2}(k_{j+1}-k_{j-1})}b_{j+1}^{-1}\] for some $\eta_j\in\Hol_\mathscr{e}(q_{{}_H}^{-1}(U),\omega)$. Crucially, note that for another path $\zeta:[0,1]\to q_{{}_H}^{-1}(U)$ with $\zeta(0)=\mathscr{e}$ and $\zeta(1)=\zeta(0)h_\zeta=\mathscr{e}h_\zeta$, we can replace $\gamma|_{[t_j,t_{j+1}]}$ with $\tilde{\alpha}^{\max(k_{j-1},k_j)}(\Rt{b_j^{-1}}(\zeta))\star\Rt{b_j h_\zeta b_j^{-1}}(\gamma|_{[t_j,t_{j+1}]})$ to change the total development of $\gamma|_{[t_j,t_{j+1}]}$ from $b_j\eta_j a^{\frac{1}{2}(k_{j+1}-k_{j-1})}b_{j+1}^{-1}$ to \[(b_j\zeta_G(1)b_j^{-1})(b_jh_\zeta^{-1}b_j^{-1})(b_j\eta_j a^{\frac{1}{2}(k_{j+1}-k_{j-1})}b_{j+1}^{-1})(b_jh_\zeta b_j^{-1}),\] which is just $b_j(\zeta_G(1)h_\zeta^{-1})\eta_ja^{\frac{1}{2}(k_{j+1}-k_{j-1})}(b_j h_\zeta b_j^{-1}b_{j+1})^{-1}$, so replacing $b_{j+1}$ with $b_j h_\zeta b_j^{-1}b_{j+1}$, every $\eta_j\in\Hol_\mathscr{e}(q_{{}_H}^{-1}(U),\omega)$ can be realized in the total development $b_j\eta_j a^{\frac{1}{2}(k_{j+1}-k_{j-1})}b_{j+1}^{-1}$ of the segment $\gamma|_{[t_j,t_{j+1}]}$ for some $\gamma$ with the given incrementation from $0$ to $0$.

Similarly, for the final segment $\gamma|_{[t_{\ell-1},1]}$ of $\gamma$, we get a path from $\gamma(t_{\ell-1})=\tilde{\alpha}^{\max(k_{\ell-2},k_{\ell-1})}(\mathscr{e})b_{\ell-1}^{-1}$ to $\gamma(1)=\gamma(0)h_\gamma=\mathscr{e}h_\gamma$, so \[(\gamma|_{[t_{\ell-1},1]})_G(1)=b_{\ell-1}\eta_{\ell-1}a^{-\max(k_{\ell-2},k_{\ell-1})}h_\gamma\] for some $\eta_{\ell-1}\in\Hol_\mathscr{e}(q_{{}_H}^{-1}(U),\omega)$. Again, by modifying the segment and $h_\gamma$, we can realize any $\eta_{\ell-1}\in\Hol_\mathscr{e}(q_{{}_H}^{-1}(U),\omega)$ in this total development of the segment.

Putting all of this together, \begin{align*}\gamma_G(1) & =(\gamma|_{[0,t_1]})_G(t_1)\cdots(\gamma|_{[t_{\ell-1},1]})_G(1) \\ & =(\eta_0 a^{k_1}b_1^{-1})(b_1\eta_1 a^{\frac{1}{2}(k_2-k_0)}b_2^{-1})\cdots(b_{\ell-1}\eta_{\ell-1} a^{-\max(k_{\ell-2},k_{\ell-1})}h_\gamma) \\ & =\eta_0 a^{k_1}\eta_1 a^{\frac{1}{2}(k_2-k_0)}\cdots \eta_{\ell-1} a^{-\max(k_{\ell-2},k_{\ell-1})}h_\gamma,\end{align*} so \[\gamma_G(1)h_\gamma^{-1}=\eta_0 a^{k_1}\eta_1 a^{\frac{1}{2}(k_2-k_0)}\cdots \eta_{\ell-1} a^{-\max(k_{\ell-2},k_{\ell-1})}.\] Note, though, that because the labels $k_j$ come from an incrementation, each of the powers of $a$ in this expression is either $a^{-1}$, $a^0=e$, or $a^1=a$, with the sum of the first $j$ powers of $a$ precisely equal to $k_j$. Moreover, since the incrementation is from $0$ to $0$, the elements $a$ and $a^{-1}$ must occur in pairs, so that $\gamma_G(1)h_\gamma^{-1}$ is in the smallest subgroup containing $\Hol_\mathscr{e}(q_{{}_H}^{-1}(U),\omega)$ closed under conjugation by powers of $a$. Thus, $\gamma_G(1)h_\gamma^{-1}$ is contained in the desired subgroup. Finally, because every $\eta_j\in\Hol_\mathscr{e}(q_{{}_H}^{-1}(U),\omega)$ can be realized in the above expression for some $\gamma$ with the given incrementation, we can get every element of the desired subgroup by considering paths $\gamma$ with different incrementations from $0$ to $0$, hence $\Hol_\mathscr{e}(\mathscr{F},\sigma^*\omega)$ is equal to this subgroup.\mbox{\qedhere}\end{proof}

\bibliographystyle{plain}
\bibliography{sprawl-refs}

\end{document}